\documentclass{article}

\usepackage{amsmath,amssymb}
\usepackage{times}
\usepackage{mathptmx}
\usepackage{epsfig}
\usepackage{bbm}
\usepackage{mathrsfs}

\newcommand{\E}{\mathbbm{E}}
\newcommand{\R}{\mathbbm{R}}

\newcommand{\Df}{\mathfrak{D}}

\newcommand{\Mf}{\mathfrak{M}}

\newcommand{\A}{\mathcal{A}}
\newcommand{\F}{\mathcal{F}}
\newcommand{\B}{\mathcal{B}}
\renewcommand{\H}{\mathcal{H}}
\newcommand{\X}{\mathcal{X}}
\newcommand{\U}{\mathcal{U}}

\newcommand{\V}{\mathcal{V}}

\newcommand{\Z}{\mathcal{Z}}
\renewcommand{\P}{\mathcal{P}}

\newcommand{\M}{\mathcal{M}}
\newcommand{\LL}{\mathcal{L}}

\newcommand{\argmin}{\mathop{\rm argmin}}

\newcommand{\graph}{\mathop{\rm graph}}
\newcommand{\varpsi}{\sl \psi}
\renewcommand{\Omega}{\varOmega}

\newtheorem{theorem}{Theorem}[section]
\newtheorem{lemma}{Lemma}[section]
\newtheorem{remark}{Remark}[section]
\newtheorem{definition}{Definition}[section]
\newtheorem{example}{Example}[section]
\newtheorem{proposition}{Proposition}[section]

\newcommand{\setdif}{\setminus}
\newcommand{\comp}{\circ}
\newcommand{\widebar}{\bar}
\newcommand{\cupprod}{\cup}

\newenvironment{proof}{\textbf{Proof.}}

\title{Risk-Averse Control of Undiscounted Transient Markov Models}
\author{\"{O}zlem \c{C}avu\c{s}\thanks{Bilkent University, Department of Industrial Engineering, Ankara, Turkey, {\tt ozlem.cavus@bilkent.edu.tr}} \and
{Andrzej Ruszczy\'nski\thanks{Rutgers University, Department of Management Science and Information Systems, Piscataway, NJ 08854, USA,
{\tt rusz@rutgers.edu}}
}}

\date{December 30, 2012}

\begin{document}
\maketitle

\begin{abstract}
We use Markov risk measures to formulate a risk-averse version of the undiscounted total cost problem
for a transient controlled Markov process. Using the new concept of a multikernel, we derive conditions for a system
to be risk-transient, that is, to have finite risk over an infinite time horizon.
We derive risk-averse dynamic programming equations satisfied by the optimal policy and we describe methods for solving these equations.
We illustrate the results on an optimal stopping problem and an organ transplant problem.
\end{abstract}

\pagestyle{myheadings}
\thispagestyle{plain}
\markboth{\"{O}. \c{C}AVU\c{S} AND A. RUSZCZY\'NSKI}{RISK-AVERSE CONTROL OF UNDISCOUNTED TRANSIENT MARKOV MODELS}

\section{Introduction}
\label{s:model}
The optimal control problem for transient Markov processes is a classical model in Operations Research (see Veinott \cite{Veinott},
Pliska \cite{Pliska}, Bertsekas and Tsitsiklis \cite{Bertsekas-Tsitsiklis}, Hernandez-Lerma and Lasser\-re~\cite{HLL2},
and the references therein). The research is focused on the expected total undiscounted cost model, with increased state and control space generality.

Our objective is to consider a risk-averse model.
So far, risk-averse problems for transient Markov models were based on the arrival probability criteria
(see, e.g., Nie and Wu \cite{Nie-Wu}
and Ohtsubo \cite{Ohtsubo}) and utility functions (see Denardo and Rothblum \cite{Denardo-Rothblum} and Patek \cite{Patek}). We plan to use the recent theory of dynamic risk measures (see Scandolo \cite{Scandolo:2003}, Riedel \cite{Riedel},
Ruszczy\'nski and Shapiro \cite{RS2005,RS2006b}, Cheridito, Delbaen and Kupper \cite{CDK:2006}, Artzner et. al. \cite{ADEHK:2007},
Kl\"oppel and Schweizer \cite{KloppelSchweizer}, Pflug and R\"omisch \cite{PflRom:07}, and the references therein)
to develop and solve new risk-averse formulations of the stochastic optimal control problem for transient Markov models.
Specific examples of such models are stochastic shortest path problems (Bertsekas and Tsitsiklis \cite{Bertsekas-Tsitsiklis})
 and optimal stopping problems (\emph{cf.} \c{C}inlar \cite{Cinlar}, Dynkin and Yushkevich \cite{DY69,DY79}, Puterman \cite{Puterman}).

Some applications of stochastic shortest path problems concerned with expected performance criteria are given
in the survey paper by White \cite{White} and the references therein. However, in many practical problems, the expected values
may not be appropriate to measure performance, because they implicitly assume that the decision maker is risk-neutral.
Below, we provide examples of such real-life
problems which were modeled before as a discrete-time Markov decision process with expected value as the objective function.
Alagoz et. al. \cite{Alagoz} suggest a discounted, infinite horizon, and absorbing Markov decision process model
to find the optimal time of liver transplant for a risk-neutral patient under the assumption that the liver is transferred from a living donor.
However, referring to Chew and Ho \cite{Chew-Ho}, they state that the risk-neutrality of the patient is not a realistic assumption.
Kurt and Kharoufe \cite{Kurt} propose a discounted, infinite horizon Markov decision process model to optimal  replacement time of a system
for a system under Markovian deterioration and Markovian environment.
So and Thomas \cite{So} employ a discrete time Markov decision process to model profitability of credit cards.

Our theory of risk-averse control problems for transient models applies to these and many other models.
Our results complement and extend the results of Ruszczy\'nski \cite{Rusz:DP},
where infinite-horizon \emph{discounted} models were considered. We consider \emph{undiscounted} models for \emph{transient}
Markov systems.
The paper is organized as follows.

In section \ref{s:CMM} we quickly review some basic concepts of controlled Markov models.
In section \ref{s:Markov} we adapt and extend our earlier theory of Markov risk measures.
In section \ref{s:multikernels} we introduce and analyze the concept of a \emph{multikernel} (a multivalued kernel),
which is essential for our theory. General assumptions and techical issues associated with measurability of decision rules
are discussed in section \ref{s:general}.
Section \ref{s:finite}
is devoted to the analysis of a finite horizon model. The main model with infinite horizon and dynamic risk measures
is analyzed  in section~\ref{s:infinite-policy}. We introduce in it the concept of a \emph{risk-transient model}
and develop equations for evaluating policies in such models.  In section  \ref{s:infinite-DP}
we derive risk-averse versions of dynamic programming equations for risk transient models.
Section \ref{s:randomized} compares randomized and deterministic polices.
Finally, section \ref{s:example} illustrates our results on  risk-averse
versions of an optimal stopping problem of Karlin \cite{Karlin}
 and  of the organ transplant problem of Alagoz \emph{et al.} \cite{Alagoz}.

\section{Controlled Markov Processes}
\label{s:CMM}

We quickly review the main concepts of controlled Markov models and
we introduce relevant notation (for details, see \cite{FS,HLL1,HLL2}).  Let ${\X}$ be a state
space, and $\U$ a control space. We assume that $\X$ and $\U$ are Borel spaces (Borel subsets of Polish spaces).
 A control set is a measurable multifunction $U:{\X}\rightrightarrows{\U}$; for each state $x\in{\X}$ the
set $U(x)\subseteq{\U}$ is a nonempty set of possible controls at $x$. A
controlled transition kernel ${Q}$ is a measurable mapping from the graph of $U$ to
the set $\P(\X)$ of probability measures on~${\X}$ (equipped with the topology of weak convergence).

The cost of transition from $x$ to $y$, when control $u$ is applied, is represented by $c(x,u,y)$, where $c:\X\times\U\times\X\to \R$.
Only $u\in U(x)$ and those $y\in {\X}$ to which transition is possible matter here, but it is convenient to consider the function $c(\cdot,\cdot,\cdot)$ as defined on the product space.

A \emph{stationary controlled Markov process} is defined by a state space ${\X}$, a
control space ${\U}$, a control set ${U}$, a controlled
transition kernel $Q$, and a cost function $c$.

For $t=1,2,\dots$ we define the space of state and control histories up to time~$t$ as
${\H}_t=\graph({U})^t\times{\X}$.
Each history is a sequence $h_t=(x_1,u_1,\dots,x_{t-1},u_{t-1},x_t)\in {\H}_t$.

We denote by $\P(\U)$ and  $\P(U(x))$ the sets of probability measures on ${\U}$ and
set of probability measures on $U(x)$.
A \emph{randomized policy} is a sequence of measurable functions $\pi_t:{\H}_t\to\P(\U)$, $t=1,2,\dots$, such that
$\pi_t(h_t) \in \P(U(x_t))$
for all $h_t\in {\H}_t$. In  words, the distribution of the control $u_t$
is supported on a subset of the set of feasible controls ${U}(x_t)$.
A \emph{Markov policy} is a sequence of measurable functions $\pi_t:{\X}\to\P({\U})$, $t=1,2,\dots$,
such that $\pi_t(x) \in \P(U(x))$ for all $x\in {\X}$. The function $\pi_t(\cdot)$ is called the \emph{decision rule} at time $t$.
A Markov policy is \emph{stationary} if there exists a function $\pi:{\X}\to\P({\U})$ such
that $\pi_t(x)=\pi(x)$, for all $t=1,2,\dots$ and all $x\in{\X}$.
Such a policy and the corresponding decision rule are called \emph{deterministic},
if for every $x\in {\X}$ there exists $u(x)\in {U}(x)$ such that the measure $\pi(x)$ is supported on $\{u(x)\}$.
In this paper we focus on deterministic policies.

Consider the canonical sample space $\Omega=\X^{\infty}$ with the product $\sigma$-algebra~$\F$.
Let $P_1$ be the initial distribution
of the state $x_1\in \X$.
Suppose we are given a  deterministic policy $\varPi=\{\pi_t\}_{t=1}^{\infty}$.
The Ionescu Tulcea theorem (see, \emph{e. g.}, \cite{BS}) states that there exists a unique probability measure $P^{\varPi}$ on
$(\Omega,\F)$
such that for every measurable set $B\subset \X$ and all $h_t\in\H_t$,  $t=1,2,\dots$,
\begin{align*}
P^{\varPi}(x_1\in B)&=P_1(B);\\
P^{\varPi}(x_{t+1}\in B \,|\, h_t) &= Q\big(B \,|\, x_t,\pi_t(h_t)\big).
\end{align*}
To simplify our notation, from now on we assume that the initial state~$x_1$ is fixed. It will be obvious
how to modify our results for a random initial state.
For a stationary decision rule $\pi$, we write $Q^{\pi}$ to denote the corresponding transition kernel.

Our interest is in \emph{transient} Markov models. We assume that some \emph{absorbing state} $x_{\rm A}\in{\X}$ exists,
such that $Q\big(\{x_{\rm A}\}\big| x_{\rm A},u\big) = 1$ and $c(x_{\rm A},u,x_{\rm A})=0$ for all $u\in {U}(x_{\rm A})$.
Thus, after the absorbing state is reached, no further costs are incurred.\footnote{The case of a larger class of absorbing states
easily reduces to the case of one absorbing state.}
To analyze such Markov models, it is convenient
to consider the effective state space $\widetilde{\X}=\X\setdif \{x_{\rm A}\}$, and the effective controlled substochastic kernel
$\widetilde{Q}$ whose arguments are restricted to $\widetilde{\X}$ and whose values are nonnegative measures on $\widetilde{\X}$, so that $\widetilde{Q}\big( B \big| x,u\big) =
{Q}\big( B \big| x,u\big)$, for all Borel sets $B\subset \widetilde{\X}$, all $x\in \widetilde{\X}$,
and all $u\in U(x)$.

Our point of departure is the \emph{expected total cost problem}, which is to
 find a policy $\varPi=\{\pi_t\}_{t=1}^{\infty}$ so as to minimize the expected cost until
absorption:
\[
\min_{\varPi}\; \E^\varPi\left[ \sum_{t=1}^{\infty} {c}(x_t,u_t,x_{t+1})\right].
\]
Here $\E^\varPi\big[ \cdot\big]$ denotes the expected value with respect to the measure $P^\varPi$.
Under appropriate assumptions, the problem has a solution
in form of a stationary Markov policy (see, e.g., \cite[sec. 9.6]{HLL2}).
The optimal policy can be found by solving appropriate dynamic programming equations.

Our intention is to introduce risk aversion to the problem,
and to replace the expected value operator by a dynamic risk measure.
We do not assume that the costs are nonnegative, and thus our approach applies also, among others, to stochastic longest path problems and optimal stopping problems with positive rewards.

\section{Markov Risk Measures}
\label{s:Markov}

Suppose $T$ is a fixed time horizon.
Each policy $\varPi=\{\pi_1,\pi_2,\dots\}$ results in a cost sequence $Z_t= c(x_{t-1},u_{t-1},x_t)$, $t=2$, $\dots$, $T+1$
on the probability space $(\varOmega,\F,P^\varPi)$.
We define the $\sigma$-subalgebras $\F_t$ on $\X^t$, and vector spaces $\Z_t^\varPi$ of $\F_t$-measurable random variables on~$\Omega$, $t=1,\dots,T$.


To evaluate risk of this sequence we  use a dynamic time-consistent risk measure of the following form:
\begin{equation}
\label{nested-2}
\begin{aligned}
J_T(\varPi,x_1) &=   \rho_1^\varPi\bigg(c(x_1,\pi_1(x_1),x_2) + \rho_2^\varPi\Big(c(x_2,\pi_2(x_2),x_3)+ \cdots \phantom{} \\
&{\quad} + \rho_{T-1}^\varPi\big(c(x_{T-1},\pi_{T-1}(x_{T-1}),x_{T})+ \rho_T^\varPi(c(x_{T},\pi_T(x_T),x_{T+1}))\big) \cdots \Big)\bigg).
\end{aligned}
\end{equation}
Here, $\rho_t^\varPi:\Z_{t+1}^\varPi\to\Z_t^\varPi$, $t=1,\dots,T$, are one-step conditional risk measures.
Ruszczy\'nski \cite[sec. 3]{Rusz:DP} derives the nested formulation \eqref{nested-2} 
from general properties of monotonicity and time-consistency of dynamic measures of risk.

It is convenient to introduce vector spaces $\Z_{t,\theta}^\varPi=\Z_t^\varPi\times\Z_{t+1}^\varPi\times \dots \times \Z_{\theta}^\varPi$,
where $1\le t \le \theta\le T+1$
and the conditional risk measures $\rho_{t,\theta}^\varPi:\Z_{t,\theta}^\varPi\to \Z_t^\varPi$ defined as follows:
\begin{equation}
\label{conditional-ext}
\rho_{t,\theta}^\varPi(Z_t,\dots,Z_{\theta}) = Z_t + \rho_t^\varPi\Big(Z_{t+1}+\rho_{t+1}^\varPi\big(Z_{t+2}+\cdots + \rho_{\theta-1}^\varPi(Z_{\theta})\cdots\big)\Big).
\end{equation}
As indicated in \cite{Rusz:DP},
the fundamental difficulty of formulation \eqref{nested-2} is that at time $t$ the value of $\rho_t^\varPi(\cdot)$ is $\F_t$-measurable and
is allowed to depend on the entire history $h_t$ of the process.
In order to overcome this difficulty, in \cite[sec. 4]{Rusz:DP}
a new construction of a one-step conditional measure of risk was introduced.
Its arguments are functions on the state space ${\X}$, rather than on the probability space~$\varOmega$.
We adapt this construction to our case, with a slightly more general form of the cost function.

Let $\V = \LL_p(\X,\B,P_0)$, where $\B$ is the $\sigma$-field of Borel sets on $\X$,   $P_0$ is some reference probability measure on $\X$,
and $p\in [1,\infty)$.
It is convenient to think of the dual space $\V'$ as the space of signed measures $m$ on $(\X,\B)$, which are
absolutely continuous with respect to $P_0$, with
densities (Radon--Nikodym derivatives) lying in the space $\LL_q(\X,\B,P_0)$, where $1/p+1/q=1$. We make the following
general assumption.
\begin{list}{}{\setlength{\rightmargin0}{\leftmargin}\itemsep0pt \topsep 5pt \itemsep 5pt}
\item[\bf G0.] For all $x\in \X$ and $u\in U(x)$ the probability measure $Q(x,u)$ is an element of $\V'$.
\end{list}
In the case of finite state and control spaces $P_0$ may be the uniform measure;
in other cases $P_0$ should be chosen in such a way that condition (G0) is satisfied.
The existence of the measure $P_0$ is essential for the pairing of $\V$ and its dual space $\V'$, as discussed below.

We consider the set of probability measures in $\V'$:
\[
\M=\left\{ m\in\V': m(\X)=1,\; m\ge 0\right\}.
\]
We also assume that the spaces $\V$ and $\V'$ are endowed with topologies that make them paired topological vector spaces
with the bilinear form
\[
\langle \varphi, m \rangle = \int_{\X} \varphi(y) \; m(dy), \quad \varphi\in \V,\quad m\in \V'.
\]
The space $\V'$ (and thus $\M$) will be endowed with the weak$^*$ topology.
We may endow $\V$ with the strong (norm) topology, or with the weak topology.

\begin{definition}
\label{d:transition}
{A measurable function
$\sigma:\V\times{\X}\times{\M}\to\R$ is a \emph{risk transition mapping}
if
for every $x\in{\X}$ and every $m\in{\M}$,  the function $\varphi\mapsto\sigma(\varphi,x,m)$
is a coherent measure of risk on $\V$.
}
\end{definition}
Recall that $\sigma(\cdot)$ is a coherent measure of risk on $\V$ (we skip the other two arguments for brevity), if (see \cite{Artzner1999})
\begin{list}{}{\setlength{\rightmargin0}{\leftmargin}\itemsep0pt \topsep 5pt \itemsep 5pt}
\item[\bf A1.] $\sigma(\alpha \varphi + (1-\alpha) \varpsi) \le \alpha \sigma(\varphi) + (1-\alpha)\sigma(\varpsi),
 \ \forall\;\alpha\in (0,1),\;\varphi,\varpsi\in\V$;
\item[\bf A2.] If  $\varphi \le \varpsi$ then $\sigma(\varphi) \le \sigma(\varpsi), \  \forall\;\varphi,\varpsi\in\V$;
\item[\bf A3.] $\sigma(a+ \varphi)=a+ \sigma(\varphi),\  \forall\; \varphi\in \V,\; a\in\R$;
\item[\bf A4.] $\sigma(\beta \varphi)=\beta \sigma(\varphi),\  \forall\;\varphi\in \V,\;\beta\ge 0$.
\end{list}

\begin{example}
\label{e:rtm-semi}
{\rm
Consider the first-order mean--semideviation risk measure analyzed by Ogryczak and Ruszczy\'n\-ski
\cite{OR1999,OR2001},  and Ruszczy\'nski and Shapiro \cite[Example 4.2]{RS2006a}, \cite[Example 6.1]{RS2006b}), but with
the state and the underlying probability measure as its arguments.  We define
\begin{equation}
\label{semideviation2}
\sigma(\varphi,x,m) = \langle \varphi,m \rangle +
\kappa \big\langle  ( \varphi - \langle \varphi,m \rangle  )_+,m \big\rangle,
\end{equation}
where $\kappa\in [0,1]$. We can verify directly that conditions (A1)--(A4) are satisfied. In a more general setting,
 $\kappa:{\X}\to [0,1]$ may be a measurable function.
}
\end{example}
\begin{example}
\label{e:rtm-cvar}
{\rm
 Another important example is the Average Value at Risk (see, \emph{inter alia}, Ogryczak and Ruszczy\'n\-ski
\cite[Sec. 4]{OR2002}, Pflug and R\"omisch \cite[Sec. 2.2.3, 3.3.4]{PflRom:07},
Rockafellar and Uryasev \cite{RU}, Ruszczy\'nski and Shapiro \cite[Example 4.3]{RS2006a}, \cite[Example 6.2]{RS2006b}), which
has the following risk transition counterpart:
\[
\sigma(\varphi,x,m) = \inf_{\eta\in\R} \bigg\{ \eta + \frac{1}{\alpha}\big\langle(\varphi-\eta)_+,m\big\rangle \bigg\},\quad \alpha\in (0,1).
\]
Again, the conditions (A1)--(A4) can be verified directly.
In a more general setting $\alpha:{\X}\to[\alpha_{\min},\alpha_{\max}]\subset(0,1)$ may be a measurable function.
}
\end{example}

We shall use the  property of \emph{law invariance} of a risk transition mapping. For a function $\varphi\in\V$
 and a probability measure $\mu\in\M$ we can define the distribution function $F_\varphi^{\mu}:\R\to [0,1]$ as follows
\[
F_\varphi^\mu(\eta) = \mu\big\{y\in \X: \varphi(y) \le \eta \big\}.
\]
\begin{definition}
\label{d:law}
A risk transition mapping $\sigma:\V\times{\X}\times{\M}\to\R$ is \emph{law invariant}, if for all
$\varphi,\varpsi \in \V$ and all $\mu,\nu\in{\M}$ such that $F_{\varphi}^{\mu} \equiv F_{\varpsi}^{\nu}$, we have
$\sigma(\varphi,x,\mu)=\sigma(\varpsi,x,\nu)$
for all $x\in {\X}$.
\end{definition}
The concept of law invariance corresponds to a similar concept for coherent measures of risk, but here we additionally need to
take into account the variability of the probability measure.
The risk transition mappings of Examples \ref{e:rtm-semi} and \ref{e:rtm-cvar} are law invariant.

The concept of law invariance is important in the context of Markov decision processes, where the model essentially defines the distribution
of the state process, for every policy $\varPi$. It also greatly simplifies the analysis of specific problems,
as illustrated in section \ref{s:house}.

Risk transition mappings allow for convenient formulation of risk-averse preferences for controlled Markov processes, where
the cost is evaluated by formula \eqref{nested-2}.
Consider a controlled Markov process $\{x_t\}$ with a deterministic Markov policy $\varPi = \{\pi_1,\pi_2,\dots\}$. For a fixed time
 $t$ and a measurable function $g:\X\times\U\times\X\to \R$  the value of $Z_{t+1} = g(x_t,u_t,x_{t+1})$ is a random variable.
 We assume that $g$ is \emph{$w$-bounded}, that is,
 \[
 \big|g(x,u,y)\big| \le C \big(w(x)+ w(y)\big),\quad \forall\; x\in \X,\ u\in U(x),\ y\in\X,
 \]
  for some constant $C>0$ and for the some \emph{weight (bounding) function} $w:\X\to [1,\infty)$, $w\in \V$
 (see, \cite[sec. 2.4]{Bauerle-Rieder}, \cite[sec. 7.2]{HLL2}, and \cite{Wessels} for the role of weight functions in Markov decision processes).
 Then $Z_{t+1}$ is an element of $\Z_{t+1}^\varPi$.
 Let $\rho_t^\varPi:\Z_{t+1}^\varPi\to\Z_t^\varPi$ be a family of conditional risk measures satisfying (A1)--(A4), for every deterministic policy
 $\varPi$. By definition,
$\rho_t^\varPi\big(g(x_t,u_t,x_{t+1})\big)$ is an element of $\Z_t^\varPi$, that is, it is an $\F_t$-measurable function on $(\varOmega,\F)$.
In the definition below, we restrict
 it to depend on the past only via the current state $x_t$.

\begin{definition}
\label{d:MarkovRisk}
{A family of one-step conditional risk measures $\rho_t^\varPi:\Z_{t+1}^\varPi\to\Z_t^\varPi$ is a \emph{Markov risk measure}
with respect to the controlled Markov process $\{x_t\}$,
if there exists a law invariant risk transition mapping $\sigma:\V\times{\X}\times{\M}\to\R$
such that for all $w$-bounded measurable functions $g:\X\times\U\times\X\to \R$
  and for all feasible deterministic Markov policies $\varPi$ we have
\begin{equation}
\label{Markov-risk}
\rho_t^\varPi\big(g(x_t,\pi_t(x_t),x_{t+1})\big) = \sigma\big(g(x_t,\pi_t(x_t),\cdot),x_t, Q(x_t,\pi_t(x_t))\big), \quad \text{a.s.}
\end{equation}
}
\end{definition}
Observe that the right hand side of formula \eqref{Markov-risk}  is parametrized by $x_t$,
and thus it defines a special $\F_t$-measurable function of $\omega$, whose dependence on the past is carried only via the
state $x_t$. The quantifier "a.s." means "almost surely with respect to the measure $P^\varPi$."

\section{Stochastic Multikernels}
\label{s:multikernels}

In order to analyze Markov measures of risk, we need to introduce the concept of a multikernel.
\begin{definition}
\label{d:multikernel}
A \emph{multikernel} is a measurable multifunction $\Mf$ from $\X$ to the space 
of regular measures on $(\X,\B(X))$.
It is \emph{stochastic}, if its values are sets of probability measures. It is \emph{substochastic},
if $0 \le M(B|x) \le 1$ for all $M\in \Mf(x)$, $B\in \B(\X)$, and $x\in \X$.
It is \emph{convex} (\emph{closed}), if for all $x\in\X$ its value $\Mf(x)$ is a convex (\emph{closed}) set.
\end{definition}

The concept of a multikernel is thus a multivalued generalization of the concept of a kernel. A measurable selector
of a stochastic multikernel $\Mf$ is a stochastic kernel $M$ such that $M(x)\in \Mf(x)$ for all $x\in \X$.
We symbolically write $M \lessdot\Mf$ to indicate that $M$ is a measurable selector of $\Mf$.

Recall that a composition $M_1  M_2$
of (sub-) stochastic  kernels $M_1$ and $M_2$ is given by the formula:
\begin{equation}
\label{kernel-comp}
\big[M_1  M_2\big]\big(B\big|x\big) = \int_\X M_2(B|y) \;M_1(dy|x),\quad \B\in \B(\X),\quad x\in \X.
\end{equation}
It is also a (sub-) stochastic  kernel.
Multikernels, in particular substochastic multikernels, can be composed in a similar fashion.

\begin{definition}
\label{d:multikernel-composition}
If $\Mf_1$ and $\Mf_2$ are multikernels,
 then their \emph{composition} $\Mf_1 \Mf_2$ is defined as follows:
\[
\big[\Mf_1 \Mf_2\big]\big(B\big|x\big)
= \Big\{ \big[M_1  M_2]\big(B\big|x\big) :  \ M_i\lessdot \Mf_i,\ i=1,2\Big\}.
\]
\end{definition}
It follows from  Definition \ref{d:multikernel-composition},
that a composition of (sub-) stochastic multikernels is a (sub-) stochastic multikernel. We may compose a substochastic multikernel $\Mf$
with itself several times, to obtain its ``power'':
\[
(\Mf)^k = \underbrace{\Mf\;  \Mf \; \cdots \; \Mf}_{k\ \text{times}}.
\]
Multikernels can be \emph{added} by employing the Minkowski sum of their values:
\[
\big[\Mf_1+\Mf_2\big](x) = \Mf_1(x)+\Mf_2(x) = \big\{ \mu: \mu=\mu_1+\mu_2,\  \mu_i\in \Mf_i(x),\ i=1,2\big\},\quad x\in \X.
\]
The sum of stochastic multikernels is a multikernel with non-negative values.

The concept of a multikernel and the composition operation arise in a natural way in the context of Markov risk measures.
If $\sigma(\cdot,\cdot,\cdot)$ is a risk transition mapping,
then the function $\sigma(\cdot,x,m)$ is lower semicontinuous for all $x\in \X$ and $m\in \M$ (see
Ruszczy\'nski and Shapiro \cite[Proposition 3.1]{RS2006a}). Then it follows from
\cite[Theorem 2.2]{RS2006a} that for every $x\in \X$ and $m\in \M$ a closed convex set $\A(x,m)\subset{\M}$ exists, such that for all $\varphi\in \V$ we have
\begin{equation}
\label{representation}
\sigma(\varphi,x,m) = \max_{\mu\in\A(x,m)} \langle \varphi,\mu \rangle.
\end{equation}
In fact, we also have
\begin{equation}
\label{subdifferential}
\A(x,m) = \partial_{\varphi} \sigma(0,x,m),
\end{equation}
that is, $\A(x,m)$ is the subdifferential of $\sigma(\cdot,x,m)$ at 0 (for the foundations of conjugate duality theory, see \cite{Rockafellar-Conjugate}).
In many cases, the multifunction $\A:\X\times\M \rightrightarrows \M$ can be described analytically.
\begin{example}
\label{e:rtm-semi2}
{\rm
For the mean-semideviation model of Example \ref{e:rtm-semi},
following the derivations of \cite[Example 4.2]{RS2006a}, we have
\begin{equation}
\label{semi-set}
\A(x,m) =\Big \{\mu\in {\M} : \exists\big(h\in \LL_{\infty}(\X,\B,P_0)\big)\; \frac{d\mu}{dm} = 1 +  h - \langle h,m\rangle,\
\|h\|_{\infty} \leq \kappa, \  h\ge  0  \Big\}.
\end{equation}
Similar formulas can be derived for higher order measures.
}
\end{example}
\begin{example}
\label{e:rtm-cvar2}
{\rm
For the Conditional Average Value at Risk of Example \ref{e:rtm-cvar},
following the derivations of \cite[Example 4.3]{RS2006a}, we obtain
\begin{equation}
\label{avar-set}
\A(x,m) =\left \{\mu\in {\M} : \frac{d\mu}{dm} \le \frac{1}{\alpha} \right\}.
\end{equation}
}
\end{example}

Consider formula \eqref{Markov-risk} and suppose that $g(x_t,u_t,x_{t+1})=v(x_{t+1})$ for some measurable
$w$-bounded function $v:\X\to\R$. Using the representation \eqref{representation} we can write it
as follows:
\begin{equation}
\label{Markov-risk2}
\rho_t^\varPi\big(v(x_{t+1})\big) = \max_{\mu\in\A\big(x_t,Q(x_t,\pi_t(x_t))\big)} \int_{\X} v(y)\; M(dy),\quad \text{a.s.}
\end{equation}
Suppose policy $\varPi$ is stationary and $\pi_t=\pi$ for all $t$. For every $x\in \X$ we can define the set of probability measures:
\begin{equation}
\label{multikernel}
\Mf^{\pi}(x) = \A\big(x,Q(x,\pi(x))\big), \quad x\in \X.
\end{equation}
The multifunction $\Mf^{\pi}:\X \rightrightarrows \P(\X)$, assigning to each $x\in\X$ the set $\Mf^{\pi}(x)$,
 is a closed  convex stochastic multikernel. We call it
a \emph{risk multikernel}, associated with the risk transition mapping $\sigma(\cdot,\cdot,\cdot)$,
the controlled kernel $Q$, and the decision rule $\pi$.
Its measurable selectors $M^{\pi}\lessdot \Mf^{\pi}$ are transition kernels.

It follows that formula \eqref{Markov-risk2} for stationary policies $\varPi$ can be rewritten as follows:
\begin{equation}
\label{rhot-M}
\rho_t^\varPi\big(v(x_{t+1})\big) = \max_{M\in {\Mf}^{\pi_{}}(x_t)} \int_{\X} v(y)\; M(dy).
\end{equation}
In the risk-neutral case we have
\[
\rho_t^\varPi\big(v(x_{t+1})\big)  = \E^\varPi\big[ v(x_{t+1})\big| x_t\big] =  \int_{\X} v(y)\; Q\big(dy\big|x_t,\pi(x_t)\big).
\]
The comparison
of the last two displayed equations reveals that in the risk-neutral case we have
\begin{equation}
\label{neutral-kernel}
\Mf^{\pi}(x) = \big\{ Q(x,\pi(x)) \big\},\quad x\in \X,
\end{equation}
that is, the risk multikernel $\Mf^{\pi}$ is single-valued, and its only selector is the kernel  $Q(\cdot,\pi(\cdot))$.
In the risk-averse case, the risk multikernel $\Mf^{\pi}$ is a closed convex-valued multifunction, whose measurable selectors
are transition kernels. It is evident that properties of this multifunction are germane for our analysis. We return to this issue
in section~\ref{s:infinite-policy}, where we calculate some examples of risk multikernels.

\begin{remark}
\label{r:lower-neutral}
If  $m\in \A(x,m)$ for all $x\in \X$ and $m\in \M$, then it follows from equation
\eqref{multikernel} that $Q(\cdot,\pi(\cdot))$ is a measurable selector of $\Mf^{\pi}$. Moreover, it follows from \eqref{representation}
that for any function $\varphi\in \V$ we have
\[
\rho_t^\varPi\big(\varphi(x_{t+1})\big) \ge \int_{\X} \varphi(y)\; Q\big(dy\big|x_t,\pi(x_t)\big)
= \E^\varPi\big[\varphi(x_{t+1})\big| x_t\big].
\]
It follows that the dynamic risk measure \eqref{nested-2} is bounded from below by the expected value of the total cost.
\end{remark}
The condition $m\in \A(x,m)$ is satisfied by the measures of risk in Examples \ref{e:rtm-semi2} and \ref{e:rtm-cvar2}.

Interestingly, uncertain transition matrices were used by Nilim and El Ghaoui in \cite{NilimElGhaoui} to increase robustness of
control rules for Markov models. There is also an intriguing connection to Markov games
(see, \emph{e.g.}, \cite{GHH,Jaskiewicz-Nowak}). In our theory, controlled multikernels arise in a natural
way in the analysis of risk-averse preferences.

\section{General Assumptions. Semicontinuity and Measurability}
\label{s:general}

We  call the controlled kernel $Q$ \emph{setwise (strongly) continuous}, if for all Borel sets $B\subset\X$
and all convergent sequences $\{(x_k,u_k)\}$, $k=1,2,\dots$,
\[
\lim_{k\to\infty}Q(B|x_k,u_k) = Q(B|x,u),
\]
where $x=\lim_{k\to\infty} x_k$ and  $u=\lim_{k\to\infty} u_k$.
We  call $Q$ \emph{weakly$^*$ continuous}, if for all functions $v\in \V$
\[
\lim_{k\to\infty}\int_\X v(y)\;Q(dy|x_k,u_k) = \int_\X v(y)\;Q(dy|x,u).
\]
Under condition (G0), setwise and weak$^*$ continuity concepts are equivalent, because the set of piecewise constant functions
is dense in $\V$.

In the product space $\X\times\M$ we always consider the product topology of
strong convergence in $\X$ and weak$^*$ convergence in $\M$.
In all our analysis we  make the following  assumptions:
\begin{list}{}{\setlength{\rightmargin0}{\leftmargin} \topsep 5pt \itemsep 5pt}
\item[\bf G1.] The transition kernel $Q(\cdot,\cdot)$ is setwise continuous;
\item[\bf G2.] The  multifunction $\A(\cdot,\cdot) \equiv \partial_{\varphi} \sigma(0,\cdot,\cdot)$ is lower semicontinuous;
\item[\bf G3.] The function $c(\cdot,\cdot,\cdot)$ is measurable, $w$-bounded, and
            $c(\cdot,\cdot,y)$ is lower semicontinuous for all $y\in\X$;
\item[\bf G4.] The multifunction $U(\cdot)$ is measurable and compact-valued.
\end{list}

We  need the following semicontinuity property of a risk transition mapping.

\begin{proposition}
\label{p:lsc}
Suppose {\rm (G0)--(G3)} and let $v\in \V$. Then the mapping
$
(x,u)\mapsto \sigma\big(c(x,u,\cdot)+v(\cdot),x, Q(x,u)\big)
$
 is lower semicontinuous on $\graph({U})$.
\end{proposition}
\begin{proof}
Let $\varphi(x,u,y)= c(x,u,y)+v(y)$. Consider the dual representation \eqref{representation} of the risk transition mapping
\begin{equation}
\label{dual-int}
\sigma(\varphi(x,u,\cdot),x, Q(x,u)) = \max_{\mu \in \A(x,Q(x,u))} \int_\X \varphi(x,u,y)\;\mu(dy).
\end{equation}
By (G0), (G1), and (G2), the multifunction $(x,u)\mapsto \A(x,Q(x,u))$
is lower semicontinuous. Owing to condition (G3), the function
$
(x,u,\mu)\mapsto  \int_\X \varphi(x,u,y)\;\mu(dy)
$
 is lower semicontinuous on $\graph({U})\times\M$.
The assertion follows now from
 \cite[Theorem 1.4.16]{Aubin-Frankowska}, whose  proof remains valid in our setting as well.
 \hfill$\Box$\\
\end{proof}

Some comments on the assumptions of Proposition~\ref{p:lsc} are in order.
Continuity assumptions of the kernel $Q$ are standard in the theory of risk-neutral Markov control processes (see, e.g.,
 \cite[App. C]{HLL1}, \cite{Schal}).
If the risk transition mapping $\sigma(\cdot,\cdot,\cdot)$ is continuous, then its subdifferential
\eqref{subdifferential} is upper semicontinuous. However, in Proposition~\ref{p:lsc} we assume \emph{lower} semicontinuity
of the mapping $(x,m)\mapsto\partial_{\varphi} \sigma(0,x,m)$, which is not trivial and should be verified for each case.
\begin{example}
\label{e:rtm-semi3}
{\rm
Let us verify the lower semicontinuity assumption for the multifunction $\A$ given in \eqref{semi-set}.
Consider an arbitrary $\mu\in \A(x,m)$ and suppose $x_k\to x$, $m_k\to m$, as $k\to\infty$.
We need to find $\mu_k\in \A(x_k,m_k)$ such that $\mu_k\to \mu$. Let $h$ be the function, for which, according to \eqref{semi-set},
$\frac{d\mu}{dm} = 1 +  h - \int h(z)\,m(dz)$. We define $\mu_k$ by specifying their Radon--Nikodym derivatives:
$\frac{d\mu_k}{dm_k} = 1 +  h - \int h(z)\,m_k(dz)$. By construction, $\mu_k\in \A(x_k,m_k)$. Then, for any function $v\in \V$ we obtain
\begin{align*}
\int_\X v(y)\;\mu_k(dy) &= \int_\X v(y)\bigg(1 +  h(y) - \int_\X h(z)\;m_k(dz)\bigg)\;m_k(dy)\\
&= \int_\X v(y)\big(1 +  h(y)\big)\;m_k(dy) - \int_\X h(z)\;m_k(dz)\;\int_\X v(y)\;m_k(dy).
\end{align*}
As $m_k\to m$, we conclude that for all $v\in \V$
\[
\lim_{k\to\infty} \int_\X v(y)\;\mu_k(dy) = \int_\X v(y)\big(1 +  h(y)\big)\;m(dy) - \int_\X h(z)\;m(dz)\;\int_\X v(y)\;m(dy) = \int_\X v(y)\;\mu(dy),
\]
which is the weak$^*$ convergence of $\mu_k$ to $\mu$.
}
\end{example}

In the following result we use the concept of a \emph{normal integrand}, that is, a function $f: \X \times U \to \R \cup {+\infty}$ such that
that its epigraphical mapping
$
x \mapsto \{ (u,\alpha)\in \U\times \R : f(x,u) \le \alpha \}
$
is a closed-valued and measurable multifunction (see Rockafellar and Wets \cite[sec. 14.D]{RW}).

\begin{proposition}
\label{p:meas}
Suppose {\rm (G0)--(G4)} and let $v\in \V$. Then the function
\[
\psi(x) = \inf_{u\in U(x)} \sigma\big(c(x,u,\cdot)+v(\cdot),x, Q(x,u)\big), \quad x\in \X,
\]
 is measurable and $w$-bounded, and a measurable selector $\pi \lessdot U$ exists, such that
 \[
\psi(x) = \sigma\big(c(x,\pi(x),\cdot)+v(\cdot),x, Q(x,\pi(x))\big), \quad \forall\;x\in \X.
  \]
\end{proposition}
\begin{proof}
Consider the function $f: \X \times U \to \R \cup {+\infty}$ defined as follows:
\[
f(x,u)= \begin{cases} \sigma\big(c(x,u,\cdot)+v(\cdot),x, Q(x,u)\big) & \text{if } u\in U(x),\\
+\infty & \text{otherwise}.
\end{cases}
\]
Owing to Proposition \ref{p:lsc}, $f(\cdot,\cdot)$  is lower semicontinuous, and is thus a normal integrand \cite[Ex. 14.31]{RW}. It follows from
\cite[Thm. 14.37]{RW} that the function $\psi(x) = \inf_u f(x,u)$ is measurable and that the optimal solution mapping
$\Psi(x) = \{u\in \U: \psi(x) = f(x,u)\}$ is measurable. By (G4), the set $U(x)$ is compact, and thus $\Psi(x)\ne \emptyset$ for all $x\in \X$.
$\Psi$ is also compact-valued.
By virtue of \cite{Kuratowski-Ryll-Nardzewski},
a measurable selector $\pi \lessdot \Psi$ exists. Let us recall the dual representation \eqref{dual-int} again:
\[
\psi(x) =  \max_{\mu \in \A(x,Q(x,\pi(x)))} \int_\X \varphi(x,\pi(x),y)\;\mu(dy),
\]
with $\varphi(x,u,y)= c(x,u,y)+v(y)$.
As the set $\A(x,Q(x,\pi(x)))$ contains only probability measures, and the function $\varphi(\cdot,\cdot,\cdot)$ is $w$-bounded, the
function $\psi(\cdot)$ is $w$-bounded as well.
\hfill$\Box$\\
\end{proof}

\section{Finite Horizon Problem}
\label{s:finite}

We consider the Markov model at times $1,2,\dots,T+1$ under deterministic policies $\varPi=\{\pi_1,\pi_2,\dots,\pi_{T}\}$.
The cost at the last stage is given by a function $v_{T+1}(x_{T+1})$. Consider the problem
\begin{equation}
\label{finite-risk}
\min_\varPi\; J_T(\varPi,x_1),
\end{equation}
with $J_T(\varPi,x_1)$ defined by formula (\ref{nested-2}), with Markov conditional risk measures
$\rho_t^\varPi$, $t=1,\dots,T$:
\begin{equation}
\label{JT}
J_T(\varPi,x_1) =  \rho_1^\varPi\bigg(c(x_1,u_1,x_2)+  \rho_2^\varPi\Big( c(x_2,u_2,x_3) + \cdots
 + \rho_{T}^\varPi\big(c(x_{T},u_{T},x_{T+1}) + v_{T+1}(x_{T+1})\big) \cdots \Big)\bigg).
\end{equation}
This means that every one-step measure has the form \eqref{Markov-risk}, with some risk transition mapping $\sigma(\cdot,\cdot,\cdot)$.

\begin{theorem}
\label{t:finiteDP}
Assume that the general conditions {\rm (G0)--(G4)} are satisfied, and that
the function $v_{T+1}(\cdot)$ is measurable and $w$-bounded.
Then
problem \eqref{finite-risk} has an optimal solution and its optimal value $v_1(x)$ is the solution of the
following {dynamic programming equations}:
\begin{equation}
v_t(x)  = \min_{u\in U(x)}\; \sigma\big(c(x,u,\cdot) + v_{t+1}(\cdot),x, Q(x,u)\big)
,\quad x\in{\X},\quad   t=T,\dots,1.\label{DP-generic}
\end{equation}
Moreover,
an optimal Markov policy $\hat{\varPi}=\{\hat{\pi}_1,\dots,\hat{\pi}_{T}\}$ exists and satisfies the equations:
\begin{equation}
\hat{\pi}_t(x)  \in \argmin_{u\in U(x)}\; \sigma\big(c(x,u,\cdot) + v_{t+1}(\cdot),x, Q(x,u)\big),
\quad x\in{\X},\quad   t=T,\dots,1.\label{DP-policy}
\end{equation}
Conversely, every  solution of equations \eqref{DP-generic}--\eqref{DP-policy} defines an optimal Markov policy
$\hat{\varPi}$.
\end{theorem}
\begin{proof}
Our proof is based on the ideas of the proof of Ruszczy\'nski \cite[Thm. 2]{Rusz:DP}, but with refinements rectifying some technical
inaccuracies.\footnote{In \cite[Thm. 2]{Rusz:DP} we missed the measurability condition on $U(\cdot)$
and the assumptions of \emph{joint} continuity (lower semicontinuity) of the kernel and the cost functions. \emph{Omnia principia parva sunt}.}

Using the monotonicity condition (A2) for $t=1,\dots,T$, we can rewrite problem (\ref{finite-risk}) as follows:
\begin{align*}
& \inf_{\pi_1,\dots,\pi_T} \bigg\{ \rho_1^\varPi\Big(c(x_1,u_1,x_2)+
\cdots  + \rho_T^\varPi\big(c(x_{T},u_T,x_{T+1})+v_{T+1}(x_{T+1}) \big) \cdots \Big) \bigg\} = \\
& \inf_{\pi_1,\dots,\pi_{T-1}} \bigg\{\rho_1^\varPi\Big(c(x_1,u_1,x_2) + \cdots
 + \inf_{\pi_T}\,\rho_{T}^\varPi\big(  c(x_T,u_T,x_{T+1})+v_{T+1}(x_{T+1}) \big) \cdots \Big) \bigg\}.
\end{align*}
Owing to the Markov structure of the conditional risk measure $\rho_T$, the innermost optimization problem
 can be rewritten as follows:
\begin{multline}
\label{DP-last-proof}
\qquad \lefteqn{\inf_{\pi_T} \sigma\big(c(x_T,\pi_T(x_T),\cdot) +v_{T+1}(\cdot),x_T,Q(x_T,\pi_T(x_T))\big) }\\
{\,} = \inf_{u\in U(x_T)} \sigma\big(c(x_T,u,\cdot) +v_{T+1}(\cdot),x_T,Q(x_T,u)\big).\qquad
\end{multline}
The problem becomes equivalent to (\ref{DP-generic}) for $t=T$, and its solution is given by (\ref{DP-policy}) for $t=T$.
By Proposi\-tion~\ref{p:meas}, a measurable selector $\hat{\pi}_T(\cdot)$ exists,
such that $\hat{\pi}_T(x_T)$ is the minimizer in (\ref{DP-last-proof}) for any $x_T$.
Finally, the optimal value in \eqref{DP-last-proof}, which we denote by $v_T(x_T)$, is measurable and $w$-bounded.

After that, the horizon $T+1$ is decreased to $T$, and the final cost becomes $v_T(x_T)$.
Proceeding in this way for $T,T-1,\dots,1$ we obtain the assertion of the theorem.
\hfill$\Box$\\
\end{proof}

It follows from our proof that the functions $v_t(\cdot)$ calculated in (\ref{DP-generic}) are the optimal values
of tail subproblems formulated for a fixed $x_t=x$ as follows:
\begin{align*}
v_t(x) &= \min_{\pi_t,\dots,\pi_T}
 \rho_t^\varPi\Big(c(x_{t},\pi_{t}(x_t),x_{t+1}) + \rho_{t+1}^\varPi\big(c(x_{t+1},\pi_{t+1}(x_{t+1}),x_{t+2}) + \cdots \\
 &{\quad} +  \rho_T^\varPi\big(c(x_T,\pi_T(x_T),x_{T+1})+v_{T+1}(x_{T+1})\big) \cdots \big) \Big).
\end{align*}
We call them \emph{value functions}, as in risk-neutral dynamic programming. It is obvious that we may have
non-stationary costs, transition kernels, and risk transition mappings in this case. Also, the assumption that the process is transient is not needed.

 Equations (\ref{DP-generic})--(\ref{DP-policy}) provide a computational recipe for solving finite horizon problems.

\section{Evaluation of Stationary Markov Policies in Infinite Horizon Problems}
 \label{s:infinite-policy}

Consider a  stationary policy $\varPi=\{\pi,\pi,\dots\}$ and define the cost until absorption as follows:
\begin{equation}
\label{Jinfty}
J_{\infty}(\varPi,x_1) = \lim_{T\to\infty}J_T(\varPi,x_1),
\end{equation}
where each $J_T(\varPi,x_1)$ is defined by the formula
\begin{equation}
\label{JT1}
\begin{aligned}
J_T(\varPi,x_1)  &=  \rho_1^\varPi\bigg(c(x_1,\pi(x_1),x_2)+ \rho_2^\varPi\Big( c(x_2,\pi(x_2),x_3)+ \cdots  + \rho_{T}^\varPi\big(c(x_{T},\pi(x_T),x_{T+1})\big) \cdots \Big)\bigg)\\
& = \rho_{1,T+1}^\varPi\big(0,c(x_1,\pi(x_1),x_2), c(x_2,\pi(x_2),x_3),\dots,c(x_{T},\pi(x_T),x_{T+1})\big),
\end{aligned}
\end{equation}
with  Markov conditional risk measures
$\rho_t^\varPi$, $t=1,\dots,T$, sharing the same risk transition mapping $\sigma(\cdot,\cdot,\cdot)$.
We assume all conditions of Theorem \ref{t:finiteDP}.

The first question to answer is when this cost is finite. This question is nontrivial, because even for uniformly bounded costs
$Z_t = c(x_{t-1},\pi(x_{t-1}),x_{t})$, $t=2,3,\dots$, and for a transient finite-state Markov chain, the limit
in \eqref{Jinfty} may be infinite, as the following example demonstrates.

\begin{example}
\label{e:counterexample}
{\rm
Consider a transient Markov chain with two states and with the following transition probabilities:
$Q_{11}=Q_{12} = \frac{1}{2}$, $Q_{22}=1$. Only one control is possible in each state, the cost of each
transition from state 1 is equal to 1, and the cost of the transition from 2 to 2 is 0. Clearly, the time until absorption is a geometric random variable with parameter $\frac{1}{2}$. Let $x_1=1$. If the limit \eqref{Jinfty} is finite, then
(skipping the dependence on $\varPi$) we have
\[
J_{\infty}(1) = \lim_{T\to\infty}J_T(1) = \lim_{T\to\infty} \rho_1\big(1 + J_{T-1}(x_2)\big) = \rho_1\big(1 + J_{\infty}(x_2)\big).
\]
In the last equation we used the continuity of $\rho_1(\cdot)$. Clearly, $J_{\infty}(2)=0$.

Suppose that we are using the Average Value at Risk  from Example \ref{e:rtm-cvar}, with $0 < \alpha \le \frac{1}{2}$,
to define $\rho_1(\cdot)$. Using standard identities for
the Average Value at Risk (see, e.g., \cite[Thm. 6.2]{SDR}), we obtain
\begin{equation}
\label{avar-comp}
\begin{aligned}
J_{\infty}(1) &= \inf_{\eta\in\R} \Big\{ \eta + \frac{1}{\alpha}\E\big[\big(1+ J_{\infty}(x_2)-\eta\big)_+\big] \Big\} \\
&= 1 + \inf_{\eta\in\R} \Big\{ \eta + \frac{1}{\alpha}\E\big[\big(J_{\infty}(x_2)-\eta\big)_+\big] \Big\}
= 1 + \frac{1}{\alpha}\int_{1-\alpha}^1 F^{-1}(\beta)\;d\beta,
\end{aligned}
\end{equation}
where $F(\cdot)$ is the distribution function of $J_{\infty}(x_2)$. As all $\beta$-quantiles of $J_{\infty}(x_2)$
for $\beta \ge \frac{1}{2}$ are equal to $J_{\infty}(1)$, the last equation yields
$J_{\infty}(1) = 1+ J_{\infty}(1)$,
a contradiction. It follows that a composition of average values at risk has no finite limit, if $0< \alpha \le \frac{1}{2}$.

On the other hand, if $\frac{1}{2}< \alpha <1$, then
\[
F^{-1}(\beta) = \begin{cases}
J_{\infty}(2)=0 & \text{if} \ 1-\alpha \le \beta < \frac{1}{2},\\
J_{\infty}(1)   & \text{if} \ \frac{1}{2} \le \beta \le 1.
\end{cases}
\]
Formula \eqref{avar-comp} then yields
$J_{\infty}(1) = 1 + \frac{1}{2\alpha}J_{\infty}(1)$.
This equation has a solution $J_{\infty}(1) = {2\alpha}/({2\alpha -1})$.

If we use the mean-semideviation model of Example \ref{e:rtm-semi}, we obtain
\begin{align*}
J_{\infty}(1) &= \E\big[1+ J_{\infty}(x_2)\big] +  \kappa\E\bigg[\Big(1+ J_{\infty}(x_2) - \E\big[1+ J_{\infty}(x_2)\big]\Big)_+ \bigg] \\
&=  1 + \frac{1}{2} J_{\infty}(1)  + \kappa \frac{1}{2} \bigg( J_{\infty}(1) -  \frac{1}{2} J_{\infty}(1) \bigg)
= 1 + \frac{2+\kappa}{4} J_{\infty}(1).
\end{align*}
Thus $J_{\infty}(1) = 4/ (2-\kappa)$, which is finite for all $\kappa\in [0,1]$, that is, for all values of $\kappa$ for which the model
defines a coherent measure of risk.
}
\end{example}

It follows that deeper properties of the measures of risk and their interplay with the transition kernel need to be investigated to
answer the question about finiteness of the dynamic measure of risk in this case.


Recall that with every risk transition mapping $\sigma(\cdot,\cdot,\cdot)$, every controlled kernel $Q$, and every decision rule
$\pi$, a multikernel $\Mf^{\pi}$ is associated, as defined in \eqref{multikernel}.
Similarly to the expected value case, it is convenient
to consider the effective state space $\widetilde{\X}=\X\setdif \{x_{\rm A}\}$, and the \emph{effective substochastic multikernel}
$\widetilde{\Mf}^{\pi}$ whose arguments are restricted to $\widetilde{\X}$
and whose values are convex sets of nonnegative measures on $\widetilde{\X}$ defined by the identity:
$\widetilde{\Mf}^{\pi}(B|x) \equiv \Mf^{\pi}(B|x)$, for all $B\in \B(\widetilde{\X})$ and $x\in \widetilde{\X}$.

A function $v\in\V$ with $v(x_{\rm A})=0$ can be identified with a function $\tilde{v}$ on $\widetilde{\X}$; we shall write $\|\tilde{v}\|$
for the norm $\|v\|$ in $\V$; we shall also write $\tilde{v}\in \V$ to indicate that the corresponding extension $v$ is an element of $\V$.
Recall that the norm $\|\cdot\|_w$ associated with a weight function $w$ is defined as follows:
\[
\|v\|_w = \sup_{x\in \widetilde{\X}}\; \frac{v(x)}{w(x)}.
\]
The corresponding operator norm $\|A\|_w$ of a substochastic kernel $A$ is defined as follows:
\[
\|A\|_w = \sup_{x\in \widetilde{\X}} \frac{1}{w(x)} \int_{\widetilde{\X}} w(y)\;A(dy|x).
\]
\begin{definition}
\label{d:risk-transient}
We call the Markov model with a risk transition mapping $\sigma(\cdot,\cdot,\cdot)$  and with a stationary Markov policy $\{\pi,\pi,\dots\}$ \emph{risk-transient}
if a weight function $w:\widetilde{\X}\to[1,\infty)$, $w\in \V$, and a constant $K$ exist such that
\begin{equation}
\label{Pliska-risk}
\big\| M \big\|_w \le K \quad
\text{for all} \quad
M \lessdot  \sum_{j=1}^T \big(\widetilde{\Mf}^{\pi_{}}\big)^j \quad \text{and all}\quad  T\ge 0.
\end{equation}
If the estimate \eqref{Pliska-risk} is uniform for all Markov policies, the model is called \emph{uniformly risk-transient}.
\end{definition}

In the special case of a risk-neutral model, owing to the equa\-tion~\eqref{neutral-kernel},
Definition \ref{d:risk-transient} reduces to the condition that
\begin{equation}
\label{Pliska}
\Big\| \sum_{j=1}^\infty \big(\widetilde{Q}^{\pi}\big)^j \Big\|_w \le K,
\end{equation}
which has been analyzed by Pliska \cite{Pliska} and \cite[sec. 9.6]{HLL2}.

\begin{example}
\label{multikernels}
{\rm
Consider the simple transient chain of Example \ref{e:counterexample} with the Average Value at Risk  from Examp\-les~\ref{e:rtm-cvar}
and \ref{e:rtm-cvar2},
where $0 < \alpha \le 1$. From \eqref{avar-set} we obtain
\[
\A(i,m) =\left \{(\mu_1,\mu_2): 0 \le \mu_j \le \frac{m_j}{\alpha},\ j=1,2;\ \mu_1+ \mu_2= 1  \right\}.
\]
As only one control is possible, formula \eqref{multikernel} simplifies to
\[
\Mf(i) = \Big\{ (\mu_1,\mu_2) : 0 \le \mu_j \le \frac{Q_{ij}}{\alpha} ,\ j=1,2;\ \mu_1+\mu_2=1\Big\},\quad i=1,2.
\]
The effective state space is just $\widetilde{\X}=\{1\}$, and we conclude that the effective multikernel is the interval
\[
\widetilde{\Mf} = \Big[0,\min\Big(1,\frac{1}{2\alpha}\Big)\Big].
\]
For $0 < \alpha \le \frac{1}{2}$ we can select $\widetilde{M}=1 \in \widetilde{\Mf}$
to show that $1\in \big(\widetilde{\Mf}\big)^j$ for all $j$, and thus condition \eqref{Pliska-risk} is not satisfied. On the other hand,
if $\frac{1}{2} < \alpha \le 1$, then for every $\widetilde{M} \in  \widetilde{\Mf}$ we have $0 \le \widetilde{M} < 1$, and condition
\eqref{Pliska-risk} is satisfied.

Consider now the mean-semideviation model of Examples \ref{e:rtm-semi} and \ref{e:rtm-semi2}. From \eqref{semi-set} we obtain
\begin{align*}
\A(i,m) &= \Big \{(\mu_1,\mu_2) :   {\mu_j} =  m_j \left(1 +  h_j - (h_1m_1+ h_2m_2)\right),\
0 \le h_j \leq \kappa,\ j=1,2\Big\},\\
\Mf(i) &= \Big \{(\mu_1,\mu_2) :  {\mu_j} = Q_{ij}\left( 1 +  h_j - (h_1Q_{i1}+ h_2Q_{i2})\right),\
0 \le h_j \leq \kappa,\ j=1,2\Big\},\quad i=1,2.
\end{align*}
Calculating the lowest and the largest possible values of $\mu_1$ we conclude that
\[
\widetilde{\Mf} = \Big[\frac{1}{2}\Big(1-\frac{\kappa}{2}\Big), \frac{1}{2}\Big(1+\frac{\kappa}{2}\Big)  \Big].
\]
For every $\kappa \in [0,1]$, Definition \ref{d:risk-transient} is satisfied.
}
\end{example}

%

We can now provide sufficient conditions for the finiteness of the limit \eqref{Jinfty}.

\begin{theorem}
\label{t:limitTk}
Suppose a stationary policy $\varPi=\{\pi,\pi,\dots\}$ is applied to a the controlled Markov model with a  risk transition mapping
$\sigma(\cdot,\cdot,\cdot)$.
If the model satisfies conditions {\rm (G0)--(G3)}, is risk-transient for the policy $\varPi$, and the cost function $c(\cdot,\cdot,\cdot)$ is $w$-bounded,
then the limit
\begin{equation}
\label{JinftyV}
J_{\infty}(\varPi,\cdot) = \lim_{T\to\infty}J_T(\varPi,\cdot),
\end{equation}
exists in $\V$ and is $w$-bounded. If the model is additionally uniformly risk-transient, then
$\big\| J_{\infty}(\varPi,\cdot)\big\|_w$ is uniformly bounded for all $\varPi$ and the
limit function $(\pi,x)\mapsto J_\infty(\varPi,x)$ is lower semicontinuous.
\end{theorem}
\begin{proof}
By Conditions (A1)--(A4), each conditional risk measure $\rho_{1,T}(\cdot)$ is convex and positively homogeneous, and thus subadditive.
For any $1 < T_1 < T_2$ we obtain the following estimate of \eqref{JT1}:
\begin{equation}
\label{subaddm}
\begin{aligned}
J_{T_2-1}(\varPi,x_1) &= {\rho_{1,T_2}^\varPi(0,Z_{2} , \dots,Z_{T_2})} \\
&\le
 {\rho_{1,T_2}^\varPi(0,Z_{2} , \dots,Z_{T_1},0,\dots,0)}+ \rho_{1,T_2}^\varPi(0 ,\dots ,0,Z_{T_1+1},\dots,Z_{T_2})\\
 &=
 {\rho_{1,T_1}^\varPi(0,Z_{2} , \dots,Z_{T_1})}+ \rho_{1,T_2}^\varPi(0 ,\dots ,0,Z_{T_1+1},\dots,Z_{T_2})\\
 &=
 J_{T_1-1}(\varPi,x_1)+ \rho_{1,T_2}^\varPi(0 ,\dots ,0,Z_{T_1+1},\dots,Z_{T_2}).
\end{aligned}
\end{equation}
As the cost function is $w$-bounded, $Z_{j+1}\le C \big(\bar{w}(x_j)+\bar{w}(x_{j+1})\big)$,
where $\bar{w}(x) = w(x)$ if $x\in \widetilde{\X}$, and $\bar{w}(x_{\text{A}}) = 0$.
Owing to the monotonicity and positive homogeneity of the conditional risk mappings,
\begin{multline}
\label{two-w}
\rho_{1,T_2}^\varPi(0 ,\dots ,0,Z_{T_1+1},\dots,Z_{T_2}) \le 2 C  \rho_{1,T_2}^\varPi(0 ,\dots ,0,\bar{w}(x_{T_1+1}),\dots,\bar{w}(x_{T_2}))  \\
 = 2 C \rho_1^\varPi\Big(\rho_2^\varPi\Big(\cdots \rho_{T_1}^\varPi\Big(\bar{w}(x_{T_1+1}) + \rho_{T_1+1}^\varPi\big(\bar{w}(x_{T_1+1})
 + \dots + \rho_{T_2-1}^\varPi\big(\bar{w}(x_{T_2})\big)\cdots\big)\Big)\cdots\Big)\Big).
\end{multline}
If $x_{T_2-1}\ne x_{\rm A}$, applying \eqref{rhot-M} to the innermost expression, we obtain
\begin{equation}
\label{rho-T-1}
\rho_{T_2-1}^\varPi\big(\bar{w}(x_{T_2})\big) = \max_{m\in {\widetilde{\Mf}}^{\pi_{}}(x_{T_2-1})} \int\limits_{\widetilde{\X}}  {w}(y)\; m(dy).
\end{equation}
It is a function of $x_{T_2-1}$, which we denote as $v_{T_2-1}(x_{T_2-1})$. Restricting the domains of the functions
to $\widetilde{\X}$,  we may write the relation:
\begin{equation}
\label{vback}
v_{T_2-1} = \widetilde{M}_{T_2-1}{w},\quad \widetilde{M}_{T_2-1} \lessdot \widetilde{\Mf}^{\pi_{}},
\end{equation}
where the selector $\widetilde{M}_{T_2-1}$ has values $\widetilde{M}_{T_2-1}(x_{T_2-1})$ equal to the maximizers in \eqref{rho-T-1}.
The maximizers exist owing to the weak$^*$ compactness of the values of the multikernel $\widetilde{\Mf}^{\pi_{}}$.
One step earlier, we obtain
\begin{equation}
\label{rho-T-2}
\begin{aligned}
\rho_{T_2-2}^\varPi\Big( \bar{w}(x_{T_2-1}) + \rho_{T_2-1}^\varPi\big(\bar{w}(x_{T_2})\big)\Big) &=
\rho_{T_2-2}^\varPi\big( \bar{w}(x_{T_2-1}) + v_{T_2-1}(x_{T_2-1})\big) \\
&= \max_{m\in \widetilde{\Mf}^{\pi_{}}(x_{T_2-2})}
\int\limits_{\X}  \big[ {w}(y) + v_{T_2-1}(y)\big]\; m(dy).
\end{aligned}
\end{equation}
Again, the maximizers $\widetilde{M}_{T_2-2}(x_{T_2-2})$ in \eqref{rho-T-2} exist, and they can be chosen in a measurable way. Denoting the
optimal value by $v_{T_2-2}(x_{T_2-2})$, we obtain a relation similar to \eqref{vback}:
\begin{equation}
\label{vback2}
v_{T_2-2} = \widetilde{M}_{T_2-2}\big( {w}+ v_{T_2-1}\big)  = \Big( \widetilde{M}_{T_2-2} + \widetilde{M}_{T_2-2} \widetilde{M}_{T_2-1}\big){w}, \quad \widetilde{M}_{T_2-2} \lessdot \widetilde{\Mf}^{\pi_{}},\quad \widetilde{M}_{T_2-1} \lessdot \widetilde{\Mf}^{\pi_{}}.
\end{equation}
Proceeding in this way, we can calculate the function
\[
v_{T_1}(x_{T_1}) =  \rho_{T_1}^\varPi\Big(\bar{w}(x_{T_1+1}) + \rho_{T_1+1}^\varPi\big(\bar{w}(x_{T_1+1})
+ \rho_{T_1+2}^\varPi\big( \bar{w}(x_{T_1+1})
 + \dots + \rho_{T_2-1}^\varPi\big(\bar{w}(x_{T_2})\big)\cdots\big)\big)\Big)
\]
on $\widetilde{\X}$ as follows:
\[
v_{T_1} = \big( \widetilde{M}_{T_1} + \widetilde{M}_{T_1}\widetilde{M}_{T_1+1} + \dots + \widetilde{M}_{T_1}\widetilde{M}_{T_1+1}\cdots\widetilde{M}_{T_2-1}\big){w},
\]
with $\widetilde{M}_j\lessdot \widetilde{\Mf}^{\pi_{}}$, $j=T_1,\dots,T_2-1$. In the formula above, we restrict the domains of the
functions to $\widetilde{X}$; at $x_{\rm{A}}$ their values are zero. Finally, defining
\begin{align*}
v_1(x_1) &= \rho_1^\varPi\Big(\rho_2^\varPi\Big(\cdots \rho_{T_1}^\varPi\Big(\bar{w}(x_{T_1+1}) + \rho_{T_1+1}^\varPi\big(\bar{w}(x_{T_1+1})
 + \dots + \rho_{T_2-1}^\varPi\big(\bar{w}(x_{T_2})\big)\cdots\big)\Big)\cdots\Big)\Big),
\end{align*}
we obtain the representation
\begin{equation}
\label{v1rep}
v_1 = \widetilde{M}_{1}\widetilde{M}_{2}\cdots\widetilde{M}_{T_1-1}\big( \widetilde{M}_{T_1} + \widetilde{M}_{T_1}\widetilde{M}_{T_1+1} + \dots + \widetilde{M}_{T_1}\widetilde{M}_{T_1+1}\dots\widetilde{M}_{T_2-1}\big){w},
\end{equation}
with $\widetilde{M}_j\lessdot \widetilde{\Mf}^{\pi_{}}$, $j=1,\dots,T_2-1$.
This combined with \eqref{subaddm}--\eqref{two-w} yields an estimate:
\begin{equation}
\label{right-estimate}
J_{T_2-1}(\varPi,\cdot) - J_{T_1-1}(\varPi,\cdot) \le
  2 C \widetilde{M}_{1}\widetilde{M}_{2}\cdots\widetilde{M}_{T_1-1}\big( \widetilde{M}_{T_1} + \widetilde{M}_{T_1}\widetilde{M}_{T_1+1} + \dots + \widetilde{M}_{T_1}\widetilde{M}_{T_1+1}\cdots\widetilde{M}_{T_2-1}\big){w}.
 \end{equation}

Consider now the sequence of costs $Z_1,\dots,Z_{T_1},-Z_{T_1+1},\dots,-Z_{T_2}$, in which we flip the sign of
the costs $Z_{t+1}=c(x_{t},u_t,x_{t+1})$ for $t \ge T_1$. From subadditivity, similarly to \eqref{subaddm}, we obtain
\begin{equation}
\label{flipped}
{\rho_{1,T_2}^\varPi(0,Z_{2} , \dots,Z_{T_1},-Z_{T_1+1},\dots, -Z_{T_2})}  \le \rho_{1,T_1}^\varPi(0,Z_{2} ,\dots , Z_{T_1})
+\rho_{1,T_2}^\varPi(0 ,\dots ,0,-Z_{T_1+1},\dots,-Z_{T_2}).
\end{equation}
By convexity of $\rho_{1,T_2}(\cdot)$,
\[
2 {\rho_{1,T_1}^\varPi(0,Z_{2} ,\dots , Z_{T_1})} \le  {\rho_{1,T_2}^\varPi(0,Z_{2} , \dots,Z_{T_1},Z_{T_1+1},\dots,Z_{T_2})} +
{\rho_{1,T_2}^\varPi(0,Z_{2} , \dots,Z_{T_1},-Z_{T_1+1},\dots,-Z_{T_2})}.
\]
Substituting the estimate \eqref{flipped}, we deduce that
\[
\rho_{1,T_2}^\varPi(0,Z_{2}, \dots,Z_{T_2}) \ge  \rho_{1,T_1}^\varPi(0,Z_{2} ,\dots , Z_{T_1})
-  \rho_{1,T_2}^\varPi(0 ,\dots ,0,-Z_{T_1+1},\dots,-Z_{T_2}).
\]

As $|Z_{t+1}|$ are bounded by $C\big(\bar{w}(x_t)+\bar{w}(x_{t+1})\big)$,  the estimate
\eqref{two-w} applies to the last element on the right hand side. We obtain
\begin{multline*}
J_{T_2-1}(\varPi,x_1) - J_{T_1-1}(\varPi,x_1)
= {\rho_{1,T_2}^\varPi(0,Z_{2}, \dots,Z_{T_2})} -  {\rho_{1,T_1}^\varPi(0,Z_{2} ,\dots , Z_{T_1})} \\
 \ge
-  2 C \rho_1^\varPi\Big(\rho_2^\varPi\Big(\cdots \rho_{T_1}^\varPi\Big(\bar{w}(x_{T_1+1}) + \rho_{T_1+1}^\varPi\big(\bar{w}(x_{T_1+1})
 + \dots + \rho_{T_2-1}^\varPi\big(\bar{w}(x_{T_2})\big)\cdots\big)\Big)\cdots\Big)\Big)
= - 2C v_1(x_1),
\end{multline*}
where $v_1(\cdot)$ has representation \eqref{v1rep}.
This combined with (\ref{right-estimate}) yields
\[
\big| J_{T_2-1}(\varPi,x_1) - J_{T_1-1}(\varPi,x_1)\big|
\le  2C  | v_1(x_1) |,\quad x_1\in \widetilde{\X}.
\]
This pointwise estimate implies the relations between the norms:
\[
\big\| J_{T_2-1}(\varPi,\cdot) - J_{T_1-1}(\varPi,\cdot)\big\|_w \le 2C  \big\| v_1 \big\|_w.
\]
In view of representation \eqref{v1rep}, we obtain the estimate
\[
\big\| J_{T_2-1}(\varPi,\cdot) - J_{T_1-1}(\varPi,\cdot)\big\|_w
\le 2C \left\| \widetilde{M}_{1}\widetilde{M}_{2}\cdots\widetilde{M}_{T_1-1}\big( \widetilde{M}_{T_1} + \widetilde{M}_{T_1}\widetilde{M}_{T_1+1} + \dots + \widetilde{M}_{T_1}\widetilde{M}_{T_1+1}\cdots\widetilde{M}_{T_2-1}\big){w} \right\|_w.
\]
By Definition \ref{d:risk-transient},
$\big\| \widetilde{M}_{T_1} + \widetilde{M}_{T_1}\widetilde{M}_{T_1+1} + \dots + \widetilde{M}_{T_1}\widetilde{M}_{T_1+1}\cdots\widetilde{M}_{T_2-1}\big\|_w \le K$, and thus
\begin{equation}
\label{J-Cauchy}
\big\| J_{T_2-1}(\varPi,\cdot) - J_{T_1-1}(\varPi,\cdot)\big\|_w
\le 2CK \big\| \widetilde{M}_{1}\widetilde{M}_{2}\cdots\widetilde{M}_{T_1-1}\big\|_w.
\end{equation}
Observe that $\widetilde{M}_{1}\widetilde{M}_{2}\cdots\widetilde{M}_{T_1-1} \lessdot \big(\widetilde{\Mf}^{\pi_{}}\big)^{T_1-1}$.
It follows from Definition \ref{d:risk-transient} that for any sequence of selectors $A_j\lessdot \big(\widetilde{\Mf}^{\pi_{}}\big)^j$
we have
$\big\| \sum_{j=1}^\infty A_j \big\|_w \le K$.
Therefore, $\big\|A_j\big\|_w \to 0$, as $j\to\infty$. Consequently,
the right hand side of \eqref{J-Cauchy} converges to 0, when $T_1,T_2\to \infty$, $T_1 < T_2$. Hence,
the sequence of functions $J_T(\varPi,\cdot)$, $T=1,2,\dots$  is  convergent to some $w$-bounded limit $J_{\infty}(\varPi,\cdot)\in \V$.
The convergence is $w$-uniform, that is,
\[
\lim_{T\to \infty} \sup_{x\in \widetilde{\X}} \frac{\big | J_T(\varPi,x) - J_\infty(\varPi,x) \big |} {w(x)}= 0.
\]
If the model is uniformly risk-transient, then the estimate \eqref{J-Cauchy} is the same for all Markov policies
$\varPi$, and thus $\big\| J_{\infty}(\varPi,\cdot)\big\|_w$ is uniformly bounded. Moreover,
\[
\lim_{T\to \infty} \sup_{{x\in \widetilde{\X}}\atop{\varPi\in \Pi^{\text{DM}}}} \frac{\big | J_T(\varPi,x) - J_\infty(\varPi,x) \big |} {w(x)}= 0,
\]
where $\Pi^{\text{DM}}$ is the set of all stationary deterministic Markov policies.
As each of the functions $(\pi,x)\mapsto J_T(\varPi,x)$ is lower semicontinuous, so is the limit function $(\pi,x)\mapsto J_\infty(\varPi,x)$.
\hfill$\Box$\\
\end{proof}
\begin{remark}
\label{r:valuePi}
It is clear from the proof of Theorem \ref{t:limitTk}, that
\begin{equation}
\label{valuePi}
J_{\infty}(\varPi,x_1) = \lim_{T\to\infty} {\rho_{1,T}^\varPi\big(0,Z_{2} , \dots,Z_{T} + f(x_T)\big)},
\end{equation}
for any $w$-bounded measurable function $f:\X\to \R$,  because $c(x_{T-1},u_t,x_T)+f(x_T)$ is still $w$-bounded.
\end{remark}

This analysis allows us to derive policy evaluation equations for the infinite horizon problem,
in the case of a fixed Markov policy.

\begin{theorem}
\label{t:DPfixed}
Suppose a  controlled Markov model with a risk transition mapping
$\sigma(\cdot,\cdot,\cdot)$ is risk-transient for the stationary Markov policy $\varPi=\{\pi,\pi,\dots\}$,
with some weight function $w(\cdot)$. If condition {\rm (G3)} is satisfied,
then a $w$-bounded function $v\in\V$ satisfies the equations
\begin{align}
v(x) &= \sigma\big(c(x,\pi(x),\cdot) + v(\cdot),x, Q(x,\pi(x))\big), \quad x\in \widetilde{\X}, \label{DPE}\\
v(x_{\rm A}) &=0, \label{DPE2}
\end{align}
if and only if $v(x) = J_{\infty}(\varPi,x)$ for all $x\in \X$.
\end{theorem}
\begin{proof}
Denote $Z_t=c(x_{t-1},u_{t-1},x_t)$.
Suppose a $w$-bounded function $v\in \V$ satisfies the equations \eqref{DPE}--\eqref{DPE2}.
By (G3),
$c(x,\pi(x),\cdot)\in \V$, and thus the right-hand side of \eqref{DPE} is well-defined.
By iteration of (\ref{DPE}), we  obtain for all $x_1\in\X$  the following equation:
\[
v(x_{1}) =  \rho_1^\varPi\bigg(c(x_1,u_1,x_2)+ \rho_2^\varPi\Big( c(x_2,u_2,x_3)+
 \cdots + \rho_{T}^\varPi\big(c(x_{T},u_{T},x_{T+1}) + v(x_{T+1})\big) \cdots \Big)\bigg).
\]
Denote $Z_t=c(x_{t-1},u_{t-1},x_t)$.
Using  subadditivity and monotonicity of the conditional risk measures we deduce that:
\begin{equation}
\label{bounds}
\begin{aligned}
v(x_{1})&=\rho_{1,T+1}^\varPi\big(0,Z_2,\dots,Z_{T+1}+v(x_{T+1})\big)
\leq \rho_{1,T+1}^\varPi\big(0,Z_2,\dots,Z_{T+1}\big) + \rho_{1,T+1}^\varPi\big(0,0,\dots,v(x_{T+1}) \big)\\
  &\leq \rho_{1,T+1}^\varPi\big(0,Z_2,\dots,Z_{T+1}\big) + \rho_{1,T+1}^\varPi\big(0,0,\dots,|v(x_{T+1})| \big).
  \end{aligned}
\end{equation}
By convexity of $\rho_{1,T+1}^\varPi(\cdot)$,
\begin{align}
2 \rho_{1,T+1}^\varPi\big(0,Z_2,\dots,Z_{T+1}\big) &\le
 \rho_{1,T+1}^\varPi\big(0,Z_2,\dots,Z_{T+1}+ v(x_{T+1})\big) +
 \rho_{1,T+1}^\varPi\big(0,Z_2,\dots,Z_{T+1}- v(x_{T+1})\big) \nonumber \\
& =
 v(x_1) +
 \rho_{1,T+1}^\varPi\big(0,Z_2,\dots,Z_{T+1}- v(x_{T+1})\big). \label{iter-convex}
\end{align}
Similar to \eqref{bounds},
\begin{align*}
\rho_{1,T+1}^\varPi\big(0,Z_2,\dots,Z_{T+1}- v(x_{T+1})\big)
&\le
\rho_{1,T+1}^\varPi\big(0,Z_2,\dots,Z_{T+1}\big) + \rho_{1,T+1}^\varPi\big(0,0,\dots,-v(x_{T+1})\big)\\
&\le
\rho_{1,T+1}^\varPi\big(0,Z_2,\dots,Z_{T+1}\big) + \rho_{1,T+1}^\varPi\big(0,0,\dots,|v(x_{T+1})|\big).
\end{align*}
Substituting into (\ref{iter-convex}) we obtain
\[
v(x_1)\ge \rho_{1,T+1}^\varPi\big(0,Z_2,\dots,Z_{T+1})) -\rho_{1,T+1}^\varPi\big(0,0,\dots,|v(x_{T+1})|\big).
\]
Combining this estimate with \eqref{bounds}, we conclude that
\begin{equation}
\label{errorJ}
\big| v(x_1) - J_T(\varPi,x_1)\big| \le \rho_{1,T+1}^\varPi\big(0,0,\dots,|v(x_{T+1})|\big).
\end{equation}
Consider the function
\[
d_{1,T}(x_1) = \rho_{1,T+1}^\varPi\big(0,0,\dots,|v(x_{T+1})|\big).
\]
Proceeding exactly as in the proof of Theorem \ref{t:limitTk}, we obtain a representation similar to \eqref{v1rep}:
\[
d_{1,T} = \widetilde{M}_{1}\cdots\widetilde{M}_{T}|v|,
\]
with $\widetilde{M}_j\lessdot \widetilde{\Mf}^{\pi_{}}$, $j=1,\dots,T_2-1$. Thus,
$d_{1,T} = A_T |v|$, with $A_T\lessdot \big(\widetilde{\Mf}^{\pi_{}}\big)^T$.
By Definition \ref{d:risk-transient}, for any sequence of selectors $A_t\lessdot \big(\widetilde{\Mf}^{\pi_{}}\big)^t$,
$t=1,2\dots,$ we have
$\big\| \sum_{t=1}^\infty A_t \big\|_w \le K$.
Therefore, $\big\|A_t\big\|_w \to 0$ and  $\big\|d_{1,t}\|_w\to 0$,  as $t\to\infty$.
Using this in \eqref{errorJ} we conclude that
$v(\cdot) \equiv J_{\infty}(\varPi,\cdot)$, as postulated.

To prove the converse implication we can use the fact that all conditional risk measures $\rho_t^\varPi(\cdot)$
share the same risk transition mapping
$\sigma(\cdot,\cdot,\cdot)$ to rewrite (\ref{JT1}) as follows:
\[
J_T(\varPi,x_1)  =  \sigma\big(c(x_1,\pi(x_1),\cdot)+ J_{T-1}(\varPi,\cdot),x_1,Q(x_1,\pi(x_1))\big).
\]
The function $\sigma(\cdot,x_1,\mu)$, as a finite-valued coherent measure of risk on a Banach lattice $\V$, is continuous
(see \cite[Prop. 3.1]{RS2006a}). By Theorem \ref{t:limitTk}, the sequence
$\big\{J_T(\varPi,\cdot)\big\}$ is convergent to $J_{\infty}(\varPi,\cdot)$ in the space $\V$, and $J_{\infty}(\varPi,\cdot)$ is $w$-bounded.
Therefore,
\[
\lim_{T\to\infty} J_T(\varPi,x_1)  =  \sigma\big(c(x_1,\pi(x_1),\cdot)+ \lim_{T\to\infty} J_{T-1}(\varPi,\cdot),x_1,Q(x_1,\pi(x_1))\big).
\]
This is identical with equation (\ref{DPE}) with $v(\cdot)\equiv J_{\infty}(\varPi,\cdot)$. Equation (\ref{DPE2}) is obvious.
\hfill$\Box$\\
\end{proof}

\section{Dynamic Programming Equations for Infinite Horizon Problems}
\label{s:infinite-DP}

We shall now focus on the optimal value function
\begin{equation}
\label{Jstar}
J^*(x) = \inf_{\varPi\in \Pi^{\text{DM}}} J_{\infty}(\varPi,x), \quad x\in \X,
\end{equation}
where $\Pi^{\text{DM}}$ is the set of all stationary deterministic Markov policies.
To simplify notation,
we define the operators $\Df:\V\rightarrow \V$ and  $\Df_{\pi}:\mathcal{V}\rightarrow \V$ as follows:
\begin{align}
&[\Df v](x) = \min_{u \in U(x)} \sigma\big(c(x,u,\cdot) + v(\cdot),x,Q(x,u)\big), \quad x \in \X, \label{Valopt}\\
&[\Df_{\pi} v] (x) = \sigma\big(c(x,\pi(x),\cdot) + v^k(\cdot),x,Q(x,\pi(x))\big) , \quad x \in \X,  \label{Valfixed}
\end{align}
where $\pi \lessdot {U}$. Owing to the monotonicity of $\sigma(\cdot,x,\mu)$, both operators are nondecreasing.
By construction, $\Df v \le \Df_{\pi} v$ for all $v\in \V$ and all $\pi\lessdot U$.

\begin{theorem}
\label{t:DPopt}
Assume that conditions {\rm (G0)--(G4)} are satisfied and that the model is uniformly risk-transient.
Then a measurable $w$-bounded function $v:\X\to\R$ satisfies the equations
\begin{align}
v(x) &= \inf_{u\in U(x)}\sigma\big(c(x,u,\cdot) + v(\cdot),x,Q(x,u)\big), \quad x\in \X, \label{DPopt1}\\
v(x_{\rm A}) &=0, \label{DPopt2}
\end{align}
if and only if $v(x) = J^*(x)$ for all $x\in \X$. Moreover, a measurable minimizer $\pi^*(x)$, $x\in \X$,
on the right hand side of \eqref{DPopt1} exists
and defines an optimal deterministic  Markov policy $\varPi^*=\{\pi^*,\pi^*,\dots\}$.
\end{theorem}
\begin{proof}
Consider a sequence of Markov deterministic policies $\varPi^k= \big\{ \pi^k,\pi^k,\dots\big\}$, $k=1,2,\dots$ constructed in the following way.
We choose any $\pi^1 \lessdot U$. Its value $v^1(\cdot) = J_{\infty}(\varPi^1,\cdot)$ is then given by equations
\eqref{DPE}--\eqref{DPE2}. For $k=1,2,\dots$ we determine $\pi^{k+1}(\cdot)$ as the measurable solution of the problem
\begin{equation}
\label{policy-iteration}
\min_{u\in U(x)} \sigma\big(c(x,u,\cdot) + v^k(\cdot),x,Q(x,u)\big), \quad x\in \X,
\end{equation}
which exists by Proposition \ref{p:meas}. The corresponding value of the policy $\varPi^{k+1}=\big\{\pi^{k+1},\pi^{k+1},\dots\big\}$
is the function $v^{k+1}(\cdot) = J_{\infty}(\varPi^{k+1},\cdot)$, and the iteration continues.
By construction, the sequences $\{\pi^k\}$ and $\{v^k\}$ satisfy the relations:
\begin{equation}
\label{basic-relations}
\Df_{\pi^{k+1}}v^{k} = \Df v^{k} \leq \Df_{\pi^{k}}v^{k}=v^{k}.
\end{equation}
Applying the operator $\Df_{\pi^{k+1}}$ to this relation,   we deduce from its monotonicity that
\begin{equation}
[\Df_{\pi^{k+1}}]^{T-1}v^{k} \leq \Df_{\pi^{k+1}}v^{k}=\Df v^{k} \leq v^{k}, \quad T=2,3,\ldots. \label{eq:PolMon}
\end{equation}
Relation \eqref{eq:PolMon} can be equivalently written as
\[
\rho_{1,T}^{\varPi^{k+1}}\Big(0,Z_2, \dots,Z_T + v^k(x_{T})\Big)
 \leq [\Df v^{k}](x_1) \leq v^{k}(x_1),
\]
where $Z_t=c(x_{t-1},u_{t-1},x_{t})$, $t=2,3,\ldots,T$, is the cost sequence resulting from the policy $\varPi^{k+1}$.
Passing to the limit with $T \rightarrow \infty$, from Remark \ref{r:valuePi} we conclude that the sequence $\{v^{k}\}$ is nonincreasing:
\begin{equation}
\label{squeeze}
v^{k+1}(x)=J_{\infty}(\varPi^{k+1},x) \leq [\Df v^{k}](x) \leq v^{k}(x), \quad x\in {\X}, \quad k=0,1,2,\ldots.
\end{equation}
Since $v^{k} \geq J^*$, the sequence $\{v^k\}$ is monotonically convergent to some limit $v^{\infty}\geq J^*$.
By Lebesgue Theorem, it is also convergent in the space $\V$. As
the function $\sigma\big(\cdot,x,\mu\big)$ is a coherent measure of risk, it follows from \cite[Prop. 3.1]{RS2006a} that
it is continuous on $\V$ and thus
\begin{equation}
\label{sigmalimit2}
\sigma\big(c(x,u,\cdot) + v^k(\cdot),x,Q(x,u)\big) \downarrow
\sigma\big(c(x,u,\cdot) + v^\infty(\cdot),x,Q(x,u)\big),\quad \text{as}\quad k\to \infty,\quad
\forall\;u\in U(x).
\end{equation}
The left inequality in \eqref{squeeze} also implies that
\begin{equation}
\label{lambdakbest2}
v^{k+1}(x)
\le \sigma\big(c(x,u,\cdot) + v^k(\cdot),x,Q(x,u)\big) ,\quad \forall\;u\in U(x).
\end{equation}
Passing to the limit with $k\to\infty$ on both sides of \eqref{lambdakbest2} and using \eqref{sigmalimit2}, we conclude that
\[
v^{\infty}(x) \le \sigma\big(c(x,u,\cdot) + v^\infty(\cdot),x,Q(x,u)\big) ,\quad \forall\;u\in U(x).
\]
Since this is true for all $x\in {\X}$ and all $u\in U(x)$, it follows that
\begin{equation}
\label{vinfty-left}
v^{\infty}(x) \le [\Df v^{\infty}](x) =  \min_{u\in U(x)} \sigma\big(c(x,u,\cdot) + v^\infty(\cdot),x,Q(x,u)\big), \quad x\in \X.
\end{equation}
Iterating this inequality, we conclude that for every feasible decision rule $\pi$ and every $T=1,2,\dots$
  \[
  v^{\infty} \le [\Df]^{T} v^{\infty} \le [\Df_\pi]^{T} v^{\infty}.
  \]
Owing to Remark \ref{r:valuePi}, the right hand side of this inequality is convergent to $J_\infty(\varPi,\cdot)$, when $T\to\infty$.
Since $\varPi$ was arbitrary, $v^{\infty} \le J^*$, and thus $v^{\infty} = J^*$. It remains to show that $v^\infty$ satisfies
equation \eqref{DPopt1}. It follows from the monotonicity of the operator $\Df$ and relation \eqref{basic-relations}, that
\[
\Df v^{\infty} \le \Df v^{k} \le v^k.
\]
Passing to the limit with $k\to\infty$ we conclude that
$\Df v^{\infty} \le v^{\infty}$.
This combined with \eqref{vinfty-left} yields $v^{\infty} = \Df v^{\infty}$, which is equation \eqref{DPopt1}. Denoting by
$\pi^*(x)$ the (measurable) minimizer on the right hand side of \eqref{vinfty-left}, we also see that $v^{\infty} = \Df_{\pi^*} v^{\infty}$.
By Theorem \ref{t:DPfixed}, $v^\infty(\cdot)=J_\infty(\varPi^*,\cdot)$.

To prove the converse implication, suppose $v\in \V$ satisfies \eqref{DPopt1}--\eqref{DPopt2} and $\|v\|_w<\infty$.
By Proposi\-tion~\ref{p:meas}, a measurable minimizer $\hat{\pi}(x)$
 on the right hand side of \eqref{DPopt1} exists.
 We obtain the equation
\[
v(x) = \sigma\big(c(x,\hat{\pi}(x),\cdot) + v(\cdot),x,Q(x,\hat{\pi}(x))\big), \quad x\in \X.
\]
Due to Theorem \ref{t:DPfixed},
\begin{equation}
\label{v-upper}
v(x) = J_{\infty}(\hat{\varPi},x)\ge J^*(x), \quad x\in \X,
\end{equation}
where $\hat{\varPi}=\{\hat{\pi},\hat{\pi},\dots\}$. On the other hand,
it follows from (\ref{DPopt1}) that for any stationary deterministic Markov  policy $\varPi=\{\pi,\pi,\dots\}$ we have
\begin{equation}
\label{ineqv}
v(x) \le \sigma\big(c(x,\pi(x),\cdot) + v(\cdot),x,Q(x,\pi(x))\big), \quad x\in \X.
\end{equation}
The risk transition mapping $\sigma$ is nondecreasing with respect to the first argument.
Therefore, iterating inequality \eqref{ineqv} we obtain the following inequality:
\[
v(x_{1}) \le \rho_{1,T}^{\varPi}\big(0,Z_{2}, \dots,Z_{T} + v(x_T)\big),
\]
Passing to the limit with $T\to\infty$ and applying Remark \ref{r:valuePi},
we obtain for all stationary deterministic Markov  policies $\varPi=\{\pi,\pi,\dots\}$ the inequality
$v(\cdot) \le J_{\infty}(\varPi,\cdot)$.
The last estimate together with (\ref{v-upper}) implies that $v(\cdot)\equiv J^*(\cdot)$ and that the
stationary policy  $\hat{\varPi}$ is optimal.
\hfill$\Box$\\
\end{proof}
\begin{remark}
\label{r:policy-iteration}
The construction of the sequence of stationary Markov policies $\{\varPi^k\}$ and their corresponding values $\{v^k\}$ employed
in the first part of the proof of Theorem \ref{t:DPopt} can be interpreted as a risk-averse version of the policy iteration method.
\end{remark}

We can now address the case of  non-stationary deterministic policies.
For a deterministic policy $\varLambda=\{\lambda_1,\lambda_2,\dots\}$ we define
\[
J_{\infty}(\varLambda,x_1) = \liminf_{T\to\infty} J_T(\varLambda,x_1)\qquad
\text{and} \qquad
\hat{J}(x_1) = \inf_{\varLambda} J_{\infty}(\varLambda,x_1).
\]
In the theorem below, we make an additional technical assumption that the function $\hat{J}(\cdot)$ is measurable.
It is obviously satisfied for a finite or countable state space. In general, its verification,
even in the expected value case, requires additional assumptions (see, e. g., \cite[sec. 9.3]{HLL2}). We discuss some sufficient conditions
after the theorem.


\begin{theorem}
\label{t:DPopt-gen}
Assume that conditions {\rm (G0)--(G4)} are satisfied and that the model is uniformly risk-transient. Additionally,
assume that the function $\hat{J}(\cdot)$ is measurable and a constant $C$ exists such that
$J_{\infty}(\varLambda,x) \ge -C w(x)$ for all $x\in \X$ and for all policies $\varLambda$.
Then a  $w$-bounded function $v\in \V$ satisfies the equations \eqref{DPopt1}--\eqref{DPopt2}
if and only if $v(x) = \hat{J}(x)$ for all $x\in \X$. Moreover, a measurable minimizer $\pi^*(x)$, $x\in \X$,
on the right hand side of \eqref{DPopt1} exists
and defines an optimal deterministic policy $\varPi^*=\{\pi^*,\pi^*,\dots\}$.
\end{theorem}
\begin{proof}
As for stationary Markov policies $\varPi$ we have $\big\|J_{\infty}(\varPi,\cdot)\big\|_w < \infty$, in view of the additional assumption
we have  $\big\| \hat{J} \big\|_w < \infty$.
Denote $\varLambda^2 = \{ \lambda_2,\lambda_3,\dots\}$. Due to the monotonicity and continuity of $\rho_1^\varPi(\cdot)$, we obtain the chain of relations
\begin{align*}
\hat{J}(x_1)  &= \inf_{\lambda_1,\lambda_2,\dots} \liminf_{T\to\infty} \rho_1^{\lambda_1}\big(c(x_1,\lambda_1(x_1),x_2)+  J_{T-1}(\varLambda^2,x_2)\big) \\
&\ge \inf_{\lambda_1,\lambda_2,\dots} \liminf_{T\to\infty} \rho_1^{\lambda_1}\big(c(x_1,\lambda_1(x_1),x_2)
+  \inf_{\tau \ge T-1} J_{\tau}(\varLambda^2,x_2)\big)\\
&= \inf_{\lambda_1,\lambda_2,\dots} \lim_{T\to\infty} \rho_1^{\lambda_1}\big(c(x_1,\lambda_1(x_1),x_2)+  \inf_{\tau \ge T-1} J_{\tau}(\varLambda^2,x_2)\big)\\
& =
\inf_{\lambda_1,\lambda_2,\dots}\rho_1^{\lambda_1}\Big(c(x_1,\lambda_1(x_1),x_2)+ \liminf_{T\to\infty} J_{T-1}(\varLambda^2,x_2)\Big) \\
&= \inf_{\lambda_1,\lambda_2,\dots}\rho_1^{\lambda_1}\big(c(x_1,\lambda_1(x_1),x_2)+ J_{\infty}(\varLambda^2,x_2)\big),
\end{align*}
Owing to the monotonicity of $\rho_1(\cdot)$, we can move the minimization with respect to $\varLambda^2$ inside:
\[
\hat{J}(x_1) \ge \inf_{\lambda_1}\rho_1^{\lambda_1}\Big(c(x_1,\lambda_1(x_1),x_2)+ \inf_{\varLambda^2} J_{\infty}(\varLambda^2,x_2)\Big) =
 \inf_{\lambda_1}\rho_1^{\lambda_1}\big(c(x_1,\lambda_1(x_1),x_2)+ \hat{J}(x_2)\big).
\]
Observe that we assumed that $\hat{J}(\cdot)$ is measurable, and thus the right hand side of the last inequality
is well-defined.
Thus $\hat{J}(\cdot)$ satisfies the inequality:
\begin{equation}
\label{DPstarx}
\hat{J}(x)\ge \inf_{u\in U(x)}\sigma\big(c(x,u,\cdot) + \hat{J}(\cdot),x, Q(x,u)\big), \quad x\in \X.
\end{equation}
We can now repeat the argument from the proof of Theorem \ref{t:DPopt}.
Denote by $\hat{\lambda}(x)$ the minimizer in \eqref{DPstarx}, which exists by Proposition \ref{p:meas}. Iterating inequality \eqref{DPstarx},
and passing to the limit we conclude that
\[
\hat{J}(x) \ge J_{\infty}(\hat{\varLambda},x),\quad x\in \X,
\]
where $\hat{\varLambda} = \{ \hat{\lambda},\hat{\lambda},\dots\}$ is a stationary Markov policy. Therefore, optimization with respect
to stationary Markov policies is sufficient, and the result follows from Theorem \ref{t:DPopt}.
\hfill$\Box$\\
\end{proof}

To prove that our additional technical assumption that $\hat{J}(\cdot)$ is measurable is
true for nonnegative costs $c(\cdot,\cdot,\cdot)$, we shall represent $\hat{J}(\cdot)$ as a limit of functions $v^k(\cdot)$ defined
by the \emph{value iteration method}:
\begin{equation}
\label{value-iteration}
v^{k+1} = \Df v^k,\quad k=0,1,2,\dots,
\end{equation}
where $v^0= 0$.
\begin{proposition}
Assume that conditions {\rm (G0)--(G4)} are satisfied, the mapping $\sigma(\cdot,x,\cdot)$ is continuous, $c(\cdot,\cdot,\cdot)$ is nonnegative,
and the model is uniformly risk-transient. Then the sequence $\{v^k\}$ defined by the value iteration method satisfies the relation
\[
\hat{J}(x) = \lim_{k\to \infty} v^k (x),\quad x\in \X.
\]
Moreover, $\hat{J}\in \V$.
\end{proposition}
\begin{proof}
Owing to the non-negativity of the costs and to the monotonicity of $\sigma(\cdot,x,m)$, the sequence $\{v^k\}$ satisfies the inequalities:
\begin{equation}
\label{values-up}
v^k \le v^{k+1} \le \hat{J} \le J^*,\quad k=0,1,2,\dots,
\end{equation}
where $J^*$ is the best value of stationary Markov policies. Hence, this sequence has a pointwise limit $v^\infty$, which is also (by Lebesgue Theorem)
a limit in $\V$. As $\sigma(\cdot,x,\cdot)$ is continuous, due to \cite[Thm. 1.4.16]{Aubin-Frankowska},  the operator $\Df$ is continuous.
 It follows that
\[
v^{\infty} = \lim_{k\to \infty} v^{k+1} =
 \lim_{k\to\infty} \Df v^k = \Df v^\infty.
\]
By Theorem \ref{t:DPopt} and relations \eqref{values-up}, $v^\infty=\hat{J}=J^*$.
\hfill$\Box$\\
\end{proof}

\section{Randomized Decision Rules}
\label{s:randomized}

Although we focus on deterministic policies in this paper, it may be of use to provide preliminary discussion
of models with randomized policies.
If randomized policies are allowed, we need to revise our definition of the risk transition mapping, to account for
the fact that for a given state $x_t=x$, \emph{both} $u_t$ and $x_{t+1}$ are random.
In this case,
 it is convenient to consider functions on the product space $\U\times\X$ equipped with its product Borel $\sigma$-algebra~$\B$.

Suppose the current state is $x$ and we use a randomized decision rule $\pi$. This control, together with the transition kernel $Q$
defines a probability measure $[\pi Q]_x$ on the product space $\U\times\X$ as follows:
\begin{equation}
\label{d:otimes}
[\pi\comp Q]_x(B_u\times B_y) = \int_{B_u} Q(B_y|x,u)\;\pi(du|x),\quad B_u\in \B(U),\quad B_y\in \B(\X).
\end{equation}
The measure is extended to other sets in $\B$ in a usual way.

The cost incurred at the current stage is given by the function  $c(x,\cdot,\cdot)$  on the product space $\U\times\X$.

The construction below parallels our earlier presentation in section \ref{s:Markov}, but with notational complication
resulting from the need to deal with the product space $\U\times\X$.
We define $\widebar{\V} = \LL_p(\U\times\X,\B,P_0)$, where $p\in [1,\infty)$, where $\widebar{P}_0$ is some reference probability measure on $\U\times \X$.
The dual space $\widebar{\V}'$ is the space of signed measures $m$ on $(\U\times\X,\B)$, which are
absolutely continuous with respect to $\widebar{P}_0$, with
densities (Radon--Nikodym derivatives) lying in the space $\LL_q(\U\times\X,\B,P_0)$, where $1/p+1/q=1$.
We consider the set of probability measures in $\widebar{\V}'$:
\[
\M=\left\{ m\in\widebar{\V}': m(\U\times \X)=1,\; m\ge 0\right\}.
\]
We also assume that the spaces $\widebar{\V}$ and $\widebar{\V}'$ are endowed with topologies that make them paired topological vector spaces
with the bilinear form
\[
\langle \varphi, m \rangle = \int_{\U\times \X} \varphi(u,y) \; m(du\times dy), \quad \varphi\in \widebar{\V},\quad m\in \widebar{\V}'.
\]

We can now define a \emph{risk transition mapping} as
a measurable function
$\sigma:\V\times{\X}\times{\M}\to\R$ , such that
for every $x\in{\X}$ and every $m\in{\M}$,  the function $\varphi\mapsto\sigma(\varphi,x,m)$
is a coherent measure of risk on $\V$.

Consider a controlled Markov process $\{x_t\}$ with some randomized Markov policy $\varPi = \{\pi_1,\pi_2,\dots\}$. For a fixed time
 $t$ and a measurable $w$-bounded function $g:\X\times\U\times\X\to \R$,  the value of $Z_{t+1} = g(x_t,u_t,x_{t+1})$,
 where $u_t$ is distributed according to $\pi_t(x_t)$,  is a random variable.

We define now a \emph{Markov one-step conditional risk measure} $\rho_t(\cdot)$ by requiring that
\[
\rho_t\big(g(x_t,u_t,x_{t+1})\big) = \sigma\big(g(x_t,\cdot,\cdot),x_t,[\pi_t \comp Q]_{x_t}\big), \quad \text{a.s.}
\]
This is analogous to \eqref{Markov-risk}.

With these modifications, we can repeat the analysis presented in sections \ref{s:Markov}--\ref{s:infinite-DP}, with formal adjustments
accounting for the use of randomized policies. For example, the dynamic programming equations \eqref{DPopt1}-\eqref{DPopt2}
would take on the form:
\begin{align}
v(x) &= \inf_{\lambda(x)\in\P({U}(x))}\sigma\big(c(x,\cdot,\cdot)+ v(\cdot),x,[\lambda\comp Q]_{x}\big), \quad x\in \X, \label{DPopt1a}\\
v(x_{\rm A}) &=0. \label{DPopt2a}
\end{align}
Justification of these equations, though, is straightforward only in the case of \emph{finite control sets} $U(x)$. Otherwise,
serious technical difficulties arise, associated with the nonlinearity of the mapping
in \eqref{DPopt1a}. With the present state of our knowledge, though, we may already make an observation
that a randomized policy may be strictly better than a deterministic policy.

Observe that the mapping $\lambda(x)\mapsto \sigma\big(c(x,\cdot,\cdot)+v(\cdot),x,[\lambda\comp Q]_{x}\big)$, which plays the key role
in the dynamic programming equation \eqref{DPopt1a}, is nonlinear, in general, as opposed to
the expected value model, where
\[
\sigma\big(c(x,\cdot,\cdot)+ v(\cdot),x,[\lambda\comp Q]_{x}\big)
= \int\limits_{U(x)} \int\limits_{\X}\big( c(x,u,y)+ v(y)\big)\; Q(dy|x,u) \;\lambda(du|x) .
\]
In the expected value case, it is sufficient to consider only the extreme points of the set $\P\big(U(x)\big)$,
which are the measures assigning unit mass to one of the controls $u\in U(x)$:
\[
 \inf_{\lambda\in\P({U}(x))} \int\limits_{U(x)} \int\limits_{\X}\big( c(x,u,y)+ v(y)\big)\; Q(dy|x,u) \;\lambda(du|x) =
 \inf_{u\in U(x)} \int\limits_{\X}\big( c(x,u,y)+ v(y)\big)\; Q(dy|x,u).
\]
 In the risk averse case, this simplification is not justified and a randomized policy may be strictly better than any
 deterministic policy.

 A question arises whether it is possible to identify cases in which deterministic policies are sufficient. It turns out that
 we can prove this for the conditional Average Value at Risk from Example \ref{e:rtm-cvar}, which in our setting takes on the
 following form:
\begin{equation}
\label{OCE}
\sigma(\varphi,x,\mu) = \inf_{\eta\in\R}
\bigg\{ \eta + \frac{1}{\alpha}\int\limits_{U(x)\times \X}\big(\varphi(u,y)-\eta\big)_+\;\mu(du\times dy) \bigg\},\quad \alpha\in (0,1).
\end{equation}
 \begin{lemma}
 If the risk transition mapping has the form \eqref{OCE} then the dynamic programming equations \eqref{DPopt1a} have a solution in deterministic decision rules.
 \end{lemma}
 \begin{proof}
 Interchanging the integration and the infimum in \eqref{OCE}, we obtain a lower bound
 \begin{align*}
 \sigma(\varphi,x,[\lambda\comp Q]_{x})
 &= \inf_{\eta\in \R} \int\limits_{U(x)} \int\limits_{\X}\big\{\eta + \big(\varphi(u,y) - \eta\big)_+\big\}\; Q(dy|x,u) \;\lambda(du|x) \\
 &\ge  \int\limits_{U(x)} \inf_{\eta\in \R}\int\limits_{\X}\big\{\eta + \big(\varphi(u,y) - \eta\big)_+\big\}\; Q(dy|x,u) \;\lambda(du|x).
 \end{align*}
 The above inequality becomes an equation for every Dirac measure $\lambda$.
 On the right-hand side of \eqref{DPopt1a} we have
 \[
\inf_{\lambda(x)\in\P({U}(x))}\sigma\big(c(x,\cdot,\cdot)+ v(\cdot),x,[\lambda\comp Q]_{x}\big)
\ge \inf_{\lambda\in\P({U}(x))}\int\limits_{U(x)} \inf_{\eta\in \R}\int\limits_{\X}\big\{\eta + \big(\varphi(u,y) - \eta\big)_+\big\}
\; Q(dy|x,u) \;\lambda(du|x).
 \]
  As the right hand side achieves its minimum over
 $\lambda\in\P({U}(x))$ at a Dirac measure concentrated at an extreme point of $U(x)$, and both sides coincide in this case,
  the minimum of the left hand side
 is also achieved at such measure. Consequently, for risk transition mappings of form \eqref{OCE} deterministic Markov
 policies are optimal.
 \hfill$\Box$\\
 \end{proof}

 In general, randomized policies may be better, as the example in section \ref{s:transplant} illustrates.

 \section{Illustrative Examples}
  \label{s:example}
  We illustrate our models and results on two examples.
 \subsection{Asset Selling}
 \label{s:house}
 Let us at first consider the classical example of asset selling originating from Karlin \cite{Karlin}.
 Offers $S_t$ arriving in time periods $t=1,2,\dots$ are independent integer-valued
 integrable random variables, distributed according to measure $P$.
 At each time we may accept the highest offer received so far, or we may wait, in which case a waiting cost $c_0$ is incurred. Denoting the
 random stopping time by $\tau$ we see that the total ``cost'' equals $Z= c_0\tau - \max_{0\le t \le \tau} S_t$.
 The problem is an example of an \emph{optimal stopping problem}, a structure of considerable theoretical and practical relevance
 (see, e.g., \c{C}inlar \cite{Cinlar}, Dynkin and Yushkevich \cite{DY69,DY79}, and Puterman \cite{Puterman}).

 Formally, we introduce the state space $\X=\{x_{\text{A}}\}\cupprod \{0,1,2,\dots\}$,
 where $x_{\text{A}}$ is the absorbing state reached after the transaction, and the other states represent the highest offer received so far.
 The control space is $\U=\{0,1\}$, with 0 representing ``wait'' and 1 representing ``sell.'' The state
 evolves according to the equation
 \[
 x_{t+1} = \begin{cases} \max( x_t, S_{t+1}) & \text{if $u_t=0$},\\
 x_{\text{A}}  & \text{if $u_t=1$}.
 \end{cases}
 \]
 Denoting by $F_S(\cdot)$ the distribution function of $S$, we can write the controlled transition kernel $Q$ as follows:
 \[
 Q(y|x,0) = \begin{cases}
 P(y) & \text{if}\ y>x,\\
 F_S(x) & \text{if}\ y=x,\\
 0 & \text{if}\ y<x,
 \end{cases}
 \qquad
 Q(y|x,0) = \begin{cases}
 1 & \text{if}\ y=x_{\text{A}},\\
 0 & \text{if}\ y\ne x_{\text{A}}.
 \end{cases}
 \]
 The cost function is
 \[
 c(x,u,y) = \begin{cases} c_0 & \text{if $u=0$},\\
 -x  & \text{if $u=1$}.
 \end{cases}
 \]
 The expected value version of this problem has a known solution: accept the first offer greater than or equal to
 the solution $\hat{x}$ of the equation
 \begin{equation}
 \label{house-expected}
 c_0 = \sum_{s=0}^{\infty} (s-\hat{x})_+\;P(s).
 \end{equation}

 We shall solve the risk-averse version of the problem.
 We choose $P_0(y)=P(y)/2$ for $y\in \mathbbm{N}$, and $P_0(x_{\text{A}})=1/2$.
 The space $\V$ is the space
 of $P_0$-integrable functions $v:\X\to \R$. The space $\M$ is the space of probability measures $\mu$ on $\X$,
 for which
 \[
 \sup\Big\{ \frac{\mu(y)}{P_0(y)} : P_0(y)>0,\ y\in \X\Big\} < \infty.
 \]
 Observe that all measures $Q(\cdot,x,u)$ are elements of $\M$.

Suppose $\sigma:\V\times\X\times\M\to\R$ is
a risk transition mapping.
Owing to (A3), and to the fact that $v(x_{\text{A}})=0$, we have
\[
\sigma\big(c(x,u,\cdot)+v(\cdot),x,Q(x,u)\big) =
\begin{cases}
c_0 + \sigma\big(v(\cdot),x,Q(x,0)\big) & \text{if}\ u=0, \\
-x                              & \text{if}\ u=1.
\end{cases}
 \]
Equation \eqref{DPopt1} takes on the form:
 \begin{equation}
 \label{DP-house}
 v(x) = \min\Big\{ -x, c_0 + \sigma\big(v(\cdot),x,Q(x,0)\big)\Big\},\quad x=0,1,2,\dots.
 \end{equation}
Suppose $\sigma$ is law invariant (Definition \ref{d:law}).
 As the distribution of $v(\cdot)$ with respect to the measure $Q(x,0)$ on $\mathbbm{N}$ is the same as the distribution
 of $v\big(\max(x,S)\big)$ under the measure $P$ of $S$, we obtain
 \[
  \sigma\big(v(\cdot),x,Q(x,0)\big) = \sigma\big(v\big(\max(x,\cdot)\big),x,P\big).
 \]
 Suppose our attitude to risk does not depend on the current state, that is, $\sigma(\cdot,\cdot,\cdot)$ does not depend on its second argument.
 Using \eqref{representation}, we may rewrite the last equation as follows:
 \[
 \sigma\big(v(\cdot),x,Q(x,0)\big) =\max_{\mu\in \A} \sum_{s=0}^{\infty} v\big(\max(x,s)\big)\; \mu(s).
 \]
 The convex closed set of probability measures $\A$ is fixed, because $\sigma(\cdot,\cdot,\cdot)$ does not depend on its second argument, and the
 third argument, $P$, is now fixed.
 Equation \eqref{DP-house} takes on the form
 \begin{equation}
 \label{DP-house-risk}
 v(x) = \min\Big\{ -x, c_0 + \max_{\mu\in \A} \sum_{s=0}^{\infty} v\big(\max(x,s)\big)\; \mu(s)\Big\},\quad x=0,1,2\dots.
 \end{equation}
 Observe that $v(x) \le -x$ and thus $v\big(\max(x,s)\big) \le -\max(x,s)$. The last displayed inequality implies that
 \[
 v(x) \le \min\Big\{ -x, c_0 + \max_{\mu\in \A} \sum_{s=0}^{\infty} \big[-\max(x,s)\big]\; \mu(s)\Big\}
 = \min\Big\{ -x, c_0 - \min_{\mu\in \A} \sum_{s=0}^{\infty} \max(x,s)\; \mu(s)\Big\},\quad x=0,1,2,\dots.
 \]
 If the offer at level $x$ is accepted, then $v(x) = -x$. After simplifications, we obtain the inequality:
 \[
 \min_{\mu\in \A} \sum_{s=0}^{\infty}(s-x)_+\;\mu(s) \le c_0.
 \]
 This suggests the solution:
 \emph{accept any offer that is greater or equal to the solution ${x}^*$ of the equation}
\begin{equation}
 \label{averse-house2}
 \min_{\mu\in \A} \sum_{s=0}^{\infty}(s-x)_+\;\mu(s) = c_0;
 \end{equation}
 \emph{if $x < {x}^*$, then wait.}
 The corresponding value function is
 $v^*(x) = -\max(x,{x}^*)$.
 Equation \eqref{DP-house-risk} can be verified by direct substitution.

Observe that the solution \eqref{averse-house2} of the risk-averse problem is closely related to the solution \eqref{house-expected}
of the expected value problem. The only difference is that we have to account for the least favorable distribution of the offers.
If $P\in \A$ (which is the typical situation), then the critical level $x^* \le \hat{x} $.

\subsection{Organ Transplant}
\label{s:transplant}
 We illustrate our results on a risk-averse version of a simplified organ transplant problem discussed in Alagoz et. al. \cite{Alagoz}.
 We consider the discrete-time absorbing Markov chain depicted in Figure~\ref{f:1}. State~S, which {is} the initial state,
 represents a patient in need of an organ transplant. State L represents life after a successful transplant.
 State D (absorbing state) represents death. Two control values
 are possible in state S: W (for ``Wait''), in which case transition to state D or back to state S may occur, and T (for ``Transplant''),
 which results in a transition to states L or D. The probability of death is lower for W than for T, but successful transplant may result in
 a longer life, as explained below. In {other two} states only one (formal) control value is possible: ``Continue''.

 \begin{figure}[h]
\centering
\vspace{-17ex}
\includegraphics[width=0.8\linewidth]{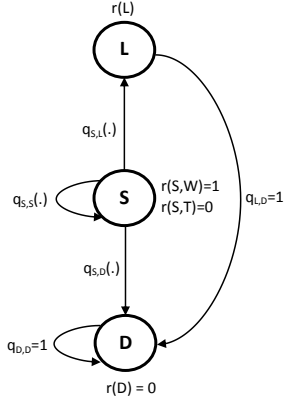}
\vspace{-52ex}
\caption{The organ transplant model.}
\label{f:1}
\end{figure}
 The rewards collected at each time step are months of life. In state S a reward equal {to} 1 is collected, if the control is W;
 otherwise, the immediate reward is 0.
 In state L the reward $r(\text{L})$ is collected, representing the sure equivalent of the random length of life after transplant.
 In state D the reward is 0.

 Generally, in a cost minimization problem, the value of a dynamic measure of risk \eqref{nested-2}
 is the ``fair'' sure charge one would be willing to incur, instead of a random sequence of costs.
 In our case, which will be a maximization problem, we shall work with the negatives of the months of life as our ``costs.'' The value of
 the measure of risk, therefore, can be interpreted as the negative of a sure life length which we consider to be equivalent to the random life duration faced by the patient.

 Let us start from describing the way the deterministic equivalent length of life $r(\text{L})$ at state L is calculated.
 The state L is in fact an aggregation
 of $n$ states in a survival model representing months of life after transplant, as depicted in Figure \ref{f:2}.

 \begin{figure}[h]
\centering
\vspace{-20ex}
\includegraphics[width=\linewidth]{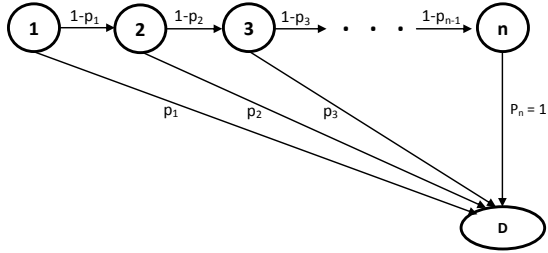}
\vspace{-36ex}
\caption{The survival model.}
\label{f:2}
\end{figure}

 In state $i=1,\dots,n$, the patient dies with
 probability $p_i$ and survives with probability $1-p_i$. The probability $p_n=1$. The reward collected at each state $i=1,\dots,n$ is equal to 1.
 In order to follow the notation of our paper, we define the cost $c(\cdot)=-r(\cdot)$.
 For illustration, we apply the mean--semideviation model of Example \ref{e:rtm-semi} with $\kappa=1$.

 The risk transition mapping has the form:
 \begin{equation}
 \label{survival-sigma}
 \sigma(\varphi,i,\nu) = \underbrace{\E_\nu[\varphi]}_{\text{expected value}} +
 \kappa \underbrace{\E_\nu\big[ \big(\varphi-\E_\nu[\varphi]\big)_+\big]}_{\text{semideviation}}.
 \end{equation}
 Owing to the monotonicity property (A2), $\sigma(\varphi,i,\nu) \le 0$, whenever $\varphi(\cdot) \le 0$.

 In \eqref{survival-sigma}, the measure $\nu$ is the transition kernel at the current state $i$,
 and the function $\varphi(\cdot)$ is the cost incurred at the current state and
 control plus the value function at the next state.
 At each state $i=1,\dots,n-1$ two transitions are possible: to D with probability $p_i$ and  $\varphi=-1$, and to $i+1$ with probability $1-p_i$ and  $\varphi=-1+v_{i+1}(i+1)$. At state $i=n$ the transition to D occurs with probability 1, and  $\varphi=-1$.
 Therefore, $v_n(n)=-1$.

 The survival problem is a finite horizon problem, and thus we apply equation  \eqref{DP-generic}.
  As there is no control to choose, the minimization operation in is eliminated.
 The equation has the form:
 \[
 v_i(i) = \sigma(\varphi,i,Q_i), \quad i=1,\dots,n-1,
 \]
 with $\varphi$ and $Q_i$ as explained above. By induction, $v_i(i)\le 0$, for $i=n-1,n-2,\dots,1$.

 Let us calculate the mean and semideviation components of \eqref{survival-sigma} at states $i=1,\dots,n-1$:
 \begin{align*}
 &\E_{Q_i}[\varphi] = -p_i + (1-p_i)\big(-1+v_{i+1}(i+1)) = -1 + (1-p_i)v_{i+1}(i+1), \\
 &\E_{Q_i}\big[ \big(\varphi-\E_{Q_i}[\varphi]\big)_+\big] = \E_{Q_i}\big[ \big(\varphi + 1 - (1-p_i)v_{i+1}(i+1)\big)_+\big] \\
 & {\qquad} = p_i\big( - 1 + 1 - (1-p_i)v_{i+1}(i+1)\big)_+ + (1-p_i)\big( -1+v_{i+1}(i+1) + 1 - (1-p_i)v_{i+1}(i+1)\big)_+ \\
  & {\qquad} = p_i\big( - (1-p_i)v_{i+1}(i+1)\big)_+ + (1-p_i)\big(p_i v_{i+1}(i+1)\big)_+ \\
   & {\qquad} = -p_i(1-p_i)v_{i+1}(i+1).
 \end{align*}
 In the last equation we used the fact that $v_{i+1}(i+1)\le 0$.
 For $i=1,\dots,n-1$, the dynamic programming equations \eqref{DP-generic} take on the form:
 \[
 v_i(i) = \underbrace{-1 + (1-p_i)v_{i+1}(i+1)}_{\text{expected value}} - \kappa \underbrace{p_i(1-p_i)v_{i+1}(i+1)}_{\text{semideviation}},
 \quad i=n-1,n-2,\dots,1.
 \]
 The value $v(1)$ is the negative of the risk-adjusted length of life with new organ. For $\kappa=0$ the above formulas give the negative of the expected
 length of life with new organ.

 In our calculations we used the transition data provided in Table \ref{t:1}. They have been chosen for purely illustrative purposes and do not correspond
 to any real medical situation.
\begin{table}[h]
\caption{ Transition probabilities from state S.}
\begin{center}
\begin{tabular}{c|ccc}
\hline
Control & S & L & D \\
\hline
W & 0.99882 &   0 & 0.00118   \\
T & 0 & 0.90782 & 0.09218\\
\hline
\end{tabular}
\end{center}
 \label{t:1}
\end{table}

For the survival model, we used the distribution function, $F(x)$, of lifetime of the American population from Jasiulewicz \cite{Jasiulewicz}.
It is a mixture of Weibull, lognormal, and Gompertz distributions:
\[
F(x)=w_{1}\Big(1-\exp\Big(-\Big(\frac{x}{\delta}\Big)^\beta\Big)\Big) + w_{2}\Phi \Big(\frac{\log x -m}{\sigma}\Big) + w_{3}\Big(1-\exp\Big(-\frac{b}{\alpha}(e^{\alpha x}-1)\Big)\Big), \quad x \geq 0.
\]
The values of the parameters and weights, provided by Jasiulewicz \cite{Jasiulewicz}, are given in Table \ref{t:2}.

\begin{table}[h]
\caption{ Values of parameters for $F(x)$.}
\begin{center}
\begin{tabular}{c|cc}
\hline
Distribution & Parameters & Weights\\
\hline
Weibull & $\delta = 0.297$, $\beta = 0.225$ &  $w_{1} = 0.0170$ \\
Lognormal & $m = 3.11$, $\sigma = 0.218$ & $w_{2} = 0.0092$ \\
Gompertz& $b = 0.0000812$, $\alpha = 0.0844$ & $w_{3} = 0.9737$ \\
\hline
\end{tabular}
\end{center}
 \label{t:2}
\end{table}

Then, we calculated the probability of dying at age $k$ (in months) as follows:
\[
p_{k} = \frac{F(k/12+1/24)-F(k/12-1/24)}{1-F(k/12-1/24)}, \quad k=1,2,\dots.
\]
The maximum lifetime of the patient was taken to be 1200 months, and that the patient after transplant has survival probabilities starting from $k=300$.
Therefore, $n=900$ in the survival model used for calculating $r(\text{L})$.

Let $\lambda=(\lambda_{_{\text{W}}},\lambda_{_{\text{T}}})$ be the randomized policy in the state S and let
$\Lambda=\big\{\lambda\in\R^2: \lambda_{_{\text{W}}}+\lambda_{_{\text{T}}} =1,\; \lambda\ge 0\big\}$.
The dynamic programming equation (\ref{DPopt1}) at S takes on the form
\begin{multline*}
v(\text{S}) =
\min_{\lambda\in\Lambda}\bigg\{
\underbrace{\lambda_{_{\text{W}}} \big[ q_{_{\text{S,S}}}(\text{W})\big(v(\text{S})-1)\big)+  q_{_{\text{S,D}}}(\text{W})\big(v(\text{D})-1\big)\big]
+ \lambda_{_{\text{T}}}\big[ q_{_{\text{S,L}}}(\text{T})v(\text{L})+  q_{_{\text{S,D}}}(\text{T})v(\text{D})\big]
}_{\text{expected value $\mu$}}\\
+ \kappa \Big(\underbrace{\lambda_{_{\text{W}}}\big[ q_{_{\text{S,S}}}(\text{W})\big(v(\text{S})-1 - \mu\big)_+  +  q_{_{\text{S,D}}}(\text{W})\big(v(\text{D})-1-\mu\big)_+ \big]
}_{\text{semideviation \dots}} \\
+ \underbrace{ \lambda_{_{\text{T}}} \big[ q_{_{\text{S,L}}}(\text{T})\big(v(\text{L}) - \mu\big)_+ +
 q_{_{\text{S,D}}}(\text{T})\big(v(\text{D}) - \mu\big)_+\big] \Big)
 }_{\text{\dots semideviation}}
\bigg\}.
\end{multline*}
 In the semideviation parts, we wrote $\mu$  for the expectation of the value
function in the next state, which is given by the first underbraced expression, and which is also dependent on $\lambda$. Of course, the above expression can be simplified,
by using the fact that $v(\text{L}) < v(\text{S}) < v(\text{D})=0$, but we prefer to leave it in the above form to illustrate the way it has been developed.

We compared two optimal control models for this problem.  The first one was the expected value model ($\kappa=0$), which corresponds to the
expected reward $r(\text{L}) = 610.46$ in the survival model. Standard dynamic programming equations were solved, and the optimal decision in state S
turned out to be W.

The second model was the risk-averse model using the mean--semideviation risk transition mapping with $\kappa=1$. This changed the
reward at state L to 515.35. We considered two versions of this model. In the first version,
we restricted the feasible policies to be deterministic. In this case, the optimal action in state
S was T. In the second version, we allowed randomized policies, as in our general model. Then the optimal policy in state S was
W with probability $\lambda_{_{\text{W}}}=0.9873$ and T with probability $\lambda_{_{\text{T}}}=0.0127$.

How can we interpret these results? The optimal randomized policy results in a random waiting time before transplanting the organ. This is
due to the fact that immediate transplant entails a significant probability of death, and a less risky policy is to ``dilute'' this
probability in a long waiting time. This cannot be derived from an expected value model, no matter what the data, because
deterministic policies are optimal in such a model: either transplant immediately or never.



{

}

\end{document}